\documentclass[11pt]{article}
\usepackage[margin=1in]{geometry}

\title{A combination theorem for Anosov subgroups}
\author{Subhadip Dey \and Michael Kapovich \and Bernhard Leeb}
\date{May 18, 2018}

\newcommand{\Addresses}{{
  \bigskip
 
\noindent
  Subhadip Dey\\ \indent
  {\footnotesize
  \textsc{Department of Mathematics\\ \indent
   University of California, Davis\\ \indent
    One Shields Avenue, Davis, CA 95616}\par\nopagebreak
  \textit{E-mail address}: \texttt{sdey@math.ucdavis.edu}}

\medskip  
\noindent Michael Kapovich\\ \indent
  {\footnotesize
  \textsc{Department of Mathematics\\ \indent
   University of California, Davis\\ \indent
    One Shields Avenue, Davis, CA 95616}\par\nopagebreak
  \textit{E-mail address}: \texttt{kapovich@math.ucdavis.edu}}
  
  \medskip\noindent
  Bernhard Leeb\\ \indent
   {\footnotesize
  \textsc{Mathematisches Institut,\\ \indent
  Universita\"{t} M\"{u}nchen, \\ \indent
  Theresienstr. 39, D-80333 Mu\"{n}chen, Germany}\par\nopagebreak
  \textit{E-mail address}: \texttt{b.l@lmu.de}}
}}

\usepackage{graphicx}
\usepackage{caption}
\usepackage{tikz}

\usepackage{mathtools}

\usepackage{hyperref}
\hypersetup{
    colorlinks=true,
    linkcolor=blue,
    filecolor=magenta,      
    urlcolor=cyan,
    citecolor=red
}

\usepackage{amsmath}
\usepackage{amsthm}
\usepackage{amssymb}
\numberwithin{equation}{section}

\begin{document}
\maketitle

\newcommand{\C}[1]{\mathrm{Cay}\left(#1\right)}
\newcommand{\flag}[1]{\mathrm{Flag}\left(#1\right)}
\newcommand{\gen}[1]{\left\langle #1 \right\rangle}
\newcommand{\hausdist}[2]{d_\mathrm{Haus} \left({#1},{#2}\right)}
\newcommand{\isom}[1]{\mathrm{Isom}\left(#1\right)}
\newcommand{\isomid}[1]{\mathrm{Isom}_0\left(#1\right)}
\newcommand{\limflag}[1]{\Lambda_\tmod\left({#1}\right)}
\newcommand{\Int}[1]{\mathrm{int}\left(#1\right)}
\newcommand{\nbhd}[2]{\mathcal{N}_{#1}\left({#2}\right)}
\newcommand{\ost}[1]{\mathrm{ost}\left(#1\right)}
\newcommand{\st}[1]{\mathrm{st}\left(#1\right)}
\newcommand{\td}[1]{\diamondsuit_\Theta\left({#1}\right)}
\newcommand{\tdd}[1]{\diamondsuit_{\Theta'}\left({#1}\right)}
\newcommand{\tddd}[1]{\diamondsuit_{\Theta''}\left({#1}\right)}
\newcommand{\tdddd}[1]{\diamondsuit_{\Theta'''}\left({#1}\right)}
\newcommand{\tmd}[1]{\diamondsuit_{\tau_{\mathrm{mod}}}\left({#1}\right)}
\newcommand{\tst}[1]{\mathrm{st}_\Theta\left(#1\right)}
\newcommand{\ttst}[1]{\mathrm{st}_{\Theta'}\left(#1\right)}
\newcommand{\tttst}[1]{\mathrm{st}_{\Theta''}\left(#1\right)}
\newcommand{\ttttst}[1]{\mathrm{st}_{\Theta'''}\left(#1\right)}
\newcommand{\V}[1]{V\left({#1}, \mathrm{st}_\Theta(\tau_{\mathrm{mod}})\right)}
\newcommand{\VV}[1]{V\left({#1}, \mathrm{st}_{\Theta'}(\tau_{\mathrm{mod}})\right)}
\newcommand{\VVt}[1]{V\left({#1}, \mathrm{st}_{\Theta'}(\tau)\right)}
\renewcommand{\vector}[1]{\overrightarrow{#1}}
\def\amod{a_\mathrm{mod}}
\def\antiflag{\left(\flag{\tmod}\times \flag{\tmod}\right)^\mathrm{opp}}
\def\Antiflag{\left(\flag{\t_+}\times \flag{\t_-}\right)^\mathrm{opp}}
\def\bg{\bar{\gamma}}
\def\bgp{\bar{\gamma}'}
\def\dopp{d^\mathrm{opp}}
\def\cat0{\mathrm{CAT}(0)}
\def\catk{\mathrm{CAT}(k)}
\def\catn{\mathrm{CAT}(-1)}
\def\di{\partial_\infty}
\def\diam{\mathrm{diam}}
\def\el{extension lemma (Proposition \ref{extension})}
\def\ev{\mathrm{ev}_{\omega_0}}
\def\fmod{{F_{\mathrm{mod}}}}
\def\fuboun{\partial_{\mathrm{F\"u}}}
\def\G{\mathfrak{g}}
\def\Ghat{\widehat{\Gamma}}
\def\g{\gamma}
\def\H{\mathbb{H}}
\def\id{\mathrm{Id}}
\def\iinv{$\iota$-invariant, Weyl-convex, compact subset}
\def\m{\mathrm{mod}}
\def\ml{Morse lemma (Theorem \ref{morse_lemma})}
\def\P{(P_+,P_-)}
\def\pt{P(\tau_-,\tau_+)}
\def\r{\rightarrow}
\def\R{\mathbb{R}}
\def\rplus{\mathbb{R}_{\ge 0}}
\def\S{\mathcal{S}}
\def\smod{{\sigma_{\mathrm{mod}}}}
\def\t{\tau}
\def\T{\Theta}
\def\tedl{triangle inequality for the $\Delta$-lengths}
\def\tmod{{\tau_{\mathrm{mod}}}}
\def\tits{\partial_{\mathrm{Tits}}}
\def\titsangle{\angle_{\mathrm{Tits}}}
\def\X{\mathcal{X}}
\theoremstyle{plain}
\newtheorem{thm}{Theorem}[section]
\newtheorem{lem}[thm]{Lemma}
\newtheorem{prop}[thm]{Proposition}
\newtheorem{cor}[thm]{Corollary}
\newtheorem{conj}[thm]{Conjecture}
\newtheorem*{claim}{Claim}
\newtheorem{ques}[thm]{Question}
\newtheorem*{case1}{Case 1}
\newtheorem*{case2}{Case 2}
\newtheorem*{case3}{Case 3}
\newtheorem*{gte}{Triangle inequality for $\Delta$-valued distances}
\theoremstyle{definition}
\newtheorem{defn}[thm]{Definition}
\newtheorem{exmp}[thm]{Example}
\theoremstyle{remark}
\newtheorem*{rem}{Remark}
\newtheorem*{note}{Note}
\newtheorem*{case}{Case}

\begin{abstract}
We prove an analogue of Klein combination theorem for Anosov subgroups by using a local-to-global principle for Morse quasigeodesics.
\end{abstract}

\tableofcontents

\section{Introduction}

The combination theorems in geometric group theory provide tools to construct new groups with ``nice'' geometric properties out of old ones. The classical Klein combination theorem \cite{klein1883neue} states that under certain assumptions the group $\gen{\Gamma_1,\Gamma_2}$ generated by two Kleinian groups $\Gamma_1$ and $\Gamma_2$ is Kleinian, and is naturally isomorphic to the  free product $\Gamma_1 * \Gamma_2$. In a series of articles \cite{maskit1965klein,maskit1968klein,maskit1971klein,maskit1993klein}, Maskit generalized the Klein combination theorem to amalgamated free products and HNN extensions. 
These so called ``Klein-Maskit combination theorems'' have been generalized to the geometrically finite subgroups of the isometry groups of higher dimensional hyperbolic spaces by several mathematicians. For instance, in \cite{baker2008combination}, Baker and Cooper proved the following theorem.

\begin{thm}[Virtual amalgam theorem, \cite{baker2008combination}]\label{thm:virtualamalgam}
If $\Gamma_1$ and $\Gamma_2$ are two geometrically finite subgroups  of $\isom{\H^n}$ which have compatible parabolic subgroups, and if $H = \Gamma_1\cap \Gamma_2$ is separable in $\Gamma_1$ and $\Gamma_2$, then there exists finite index subgroups $\Gamma_1'$ and $\Gamma_2'$ of $\Gamma_1$ and $\Gamma_2$, respectively, containing $H$ such that the group $\gen{\Gamma_1',\Gamma_2'}$ generated by $\Gamma_1'$ and $\Gamma_2'$ is geometrically finite, and is naturally isomorphic to the  amalgam $\Gamma_1' *_H \Gamma_2'$.
\end{thm}

When $\Gamma_1$ and $\Gamma_2$ intersect trivially, the ``compatibility condition'' in the above theorem simply means that the limit sets of $\Gamma_1$ and $\Gamma_2$ in $\di\H^n$ are disjoint. Since this case would be most relevant to our work, we state it separately.

\begin{cor}\label{thm:virtualamalgam2}
If $\Gamma_1$ and $\Gamma_2$ are two geometrically finite subgroups  of $\isom{\H^n}$ with disjoint limit sets in $\di\H^n$, then there exists finite index subgroups $\Gamma_1'$ and $\Gamma_2'$ of $\Gamma_1$ and $\Gamma_2$, respectively, such that the group $\gen{\Gamma_1',\Gamma_2'}$ generated by $\Gamma_1'$ and $\Gamma_2'$ is geometrically finite and is naturally isomorphic $\Gamma_1' * \Gamma_2'$.
\end{cor}

There are also certain generalizations of these combination theorems in the realm of subgroups of hyperbolic groups and, more generally, isometry groups of Gromov-hyperbolic spaces. In \cite{GITIK199965}, Gitik proved that under certain conditions two quasiconvex subgroups of a $\delta$-hyperbolic group ``can be virtually amalgamated.''  In this regard, our main theorem is an analogue of \cite[Corollary 3]{GITIK199965}. See also the papers by Mart\'inez-Pedroza \cite{pedroza2009} and Mart\'inez-Pedroza--Sisto \cite{MR2994828} for closely related results.

\medskip
In the present work, we prove a combination theorem for Anosov subgroups  of semisimple Lie groups.  {Anosov representations} of surface groups (and, more generally, fundamental groups of compact negatively curved manifolds) were introduced by Labourie \cite{labourie2006anosov} to study the ``Hitchin component'' of the space of reducible representations in $PSL(n,\R)$. 
Guichard and Weinhard \cite{guichard2012anosov} generalized this notion in the setting of representations of  hyperbolic groups in real semisimple Lie groups. Anosov subgroups can be regarded as higher rank generalizations of convex-cocompact subgroups of isometry groups of negatively curved symmetric spaces. 

Our main result presents an analogue of Corollary \ref{thm:virtualamalgam2} for Anosov subgroups. Before stating our theorem, we briefly discuss our framework. Let $G$ be a semisimple Lie group, let $P$ be a maximal parabolic subgroup conjugate to its opposite subgroups.

Our main result is the following.

\begin{thm}[Combination theorem] \label{thm:combination}
Let $\Gamma_1,\dots,\Gamma_n$ be pairwise antipodal, residually finite\footnote{It suffices to assume that each $\Gamma_i$ has trivial intersection with the center of $G$. See also the remark following Theorem \ref{main_result}.}
 $P$-Anosov subgroups of $G$. Then there exist finite index subgroups $\Gamma'_i$ of $\Gamma_i$, for $i=1,\dots,n$, such that the subgroup $\gen{\Gamma'_1,\dots,\Gamma'_n}$ generated by $\Gamma'_1,\dots,\Gamma'_n$ in $G$ is $P$-Anosov, and is naturally isomorphic to the  free product $\Gamma'_1*\dots*\Gamma'_n$.
\end{thm}

The undefined term ``antipodal'' will be made precise later in the paper (Definition \ref{def:antipodal_morse}): This condition replaces the disjointness of the limit sets in Corollary \ref{thm:virtualamalgam2}. Moreover, the geometric finiteness in Corollary \ref{thm:virtualamalgam2} is replaced by the Anosov condition.

In fact, our combination theorem is a special case of a more geometric result (Theorem \ref{thm:general}), stated in terms of ``sufficiently high displacements'' and ``sufficient antipodality'' of the groups $\Gamma_i$ at a point $x\in X=G/K$; with this condition, there is no need to pass to finite index subgroups.

\medskip
Although we state main result (Theorem \ref{thm:combination}) in the language of Anosov representations, we never really use it in our proof. Instead, we use the language of \emph{Morse subgroups}, and prove an equivalent statement in this context (Theorem \ref{main_result}).

In \cite{Kapovich:2014aa}, Kapovich, Leeb and Porti introduced a class of discrete subgroups of isometries of higher rank symmetric spaces. This class of subgroups generalizes the convex cocompact subgroups in the rank one Lie groups. In \cite{Kapovich:2014aa} and in subsequent articles  \cite{Kapovich:2014ab,MR3736790,MR3720343}, they introduced and proved several equivalent definitions of this class, and studied their geometric properties (e.g. structural stability, cocompactness etc.). Some of these equivalent definitions are given in terms of RCA subgroups, URU subgroups, Morse subgroups, asymptotically embedded subgroups,  etc. In \cite{Kapovich:2014aa}, they also proved that the classes of Morse subgroups and Anosov subgroups are equal. 

\begin{thm}[Morse $\Leftrightarrow$ Anosov, \cite{Kapovich:2014aa}] \label{equiv_RCA} For a discrete subgroup $\Gamma$ of $G$, the following are equivalent.
\begin{enumerate}
\item $\Gamma$ is $P_{\tmod}$-Anosov.
\item $\Gamma$ is $\tmod$-Morse.
\end{enumerate}
\end{thm}

See also \cite{Kapovich:2017aa} and \cite{survey} for detailed surveys on these results.

\medskip 
In  \cite[Theorem 7.40]{Kapovich:2014aa} they used the local-to-global principle for 
Morse quasigeodesics to construct (free) Morse-Schottky subgroups of semisimple Lie groups (cf. also \cite{Benoist}):

\begin{thm}
Suppose that $g_1,...,g_n$ are hyperbolic isometries of a symmetric space $X=G/K$ of noncompact type, whose repelling/attracting points in the flag-manifold $G/P_{\tau_{mod}}$ are pairwise antipodal. Then for all sufficiently large $N$, the subgroup of $G$ generated by $g_1^N,...,g_n^N$ is $\tau_{mod}$-Morse and free of rank $k$. 
\end{thm}

While our main theorem contains this result as a special case 
when the subgroups $\Gamma_1,...,\Gamma_n$ are cyclic, our proof uses some of the main ideas of the proof of \cite[Theorem 7.40]{Kapovich:2014aa}. 

\subsection*{Organization of this paper}
In section \ref{sec:preliminaries}, we give a brief overview on symmetric spaces of noncompact type, $\Delta$-valued distances and the triangle inequalities, $\tmod$-regularities, parallel sets, $\xi$-angles, $\T$-cones, and $\T$-diamonds, mostly to set up our notations while leaving the details to the references. 
In section \ref{estimates}, we prove several estimates on $\xi$-angles which will provide crucial ingredients for construction of Morse embeddings in the proof of our main result. In section \ref{sec:morse}, or more specifically in \ref{stability} and \ref{sec:morsesubgroups},  we discuss Morse properties. In section \ref{sec:replacements}, we introduce the replacement lemma (Theorem \ref{replacements}, and a generalized version Theorem \ref{replacements2}) which is another important ingredient in the proof of our main result. In section \ref{sec:RF}, we discuss the residual finiteness property of Morse subgroups. Finally, in section \ref{sec:combination}, we state and prove our main result in terms of  Morse subgroups (Theorem \ref{main_result}).

\subsection*{Notations}
Here is a list of commonly used notations.
\begin{itemize}
\item $\angle^\xi_x(x_1,x_2)$ = $\xi$-angle between $\tmod$-regular segments $xx_1$ and $xx_2$ (see section \ref{sec:boundary})
\item $\td{x_1,x_2}$ = $\T$-diamond with tips at $x_1$ and $x_2$ (see section \ref{sec:parallelset})
\item $\iota$ = the opposition involution (see section \ref{sec:symm_noncompact})
\item $\nbhd{D}{\cdot}$ =  open $D$-neighborhood
\item $\ost{\tau}$ = open star of $\tau$ in the visual boundary (see section \ref{sec:parallelset})
\item $\st{\tau}$ = star of $\tau$ in the visual boundary (see section \ref{sec:parallelset})
\item $V(x, \tst{\tau})$ = $\T$-cone asymptotic to $\tau$ with tip at $x$ (see section \ref{sec:parallelset})
\end{itemize}

\medskip 
{\bf Acknowledgements.} The second author was partly supported by the NSF grant  DMS-16-04241, by 
KIAS (the Korea Institute for Advanced Study) through the KIAS scholar program, by a Simons Foundation Fellowship, grant number 391602, and by Max Plank Institute for Mathematics in Bonn. 

\section{Geometric background}\label{sec:preliminaries}
In this section, we first review some notions pertinent to geometry of  \emph{symmetric spaces of noncompact type}. A standard reference for this section is \cite{eberlein1996geometry}. Then we briefly review various notions such as ideal boundaries, Tits metrics, $\tmod$-regularity, $\T$-cones, $\T$-diamonds etc. enough to fix our notations and conventions. For a detailed exposition on these topics, we refer to \cite{Kapovich:2014aa,Kapovich:2014ab}.

\subsection{Symmetric spaces of non-compact type}\label{sec:symm_noncompact}

A \emph{(global) symmetric space} $X$ is a Riemannian manifold which has an \emph{inversion symmetry} about each point, i.e. for each point $x\in X$, there exists an isometric involution $s_x : X\rightarrow X$ fixing $x$, called the \emph{Cartan involution}, whose differential $ds_x$ restricts to $-\mathrm{Id}$ on the tangent space $T_x X$. In this paper we consider only simply connected symmetric spaces.

Each symmetric space has a de Rham decomposition into irreducible symmetric spaces. A symmetric space $X$ is said to be of \emph{noncompact type} if it is nonpositively curved, simply connected and without a Euclidean factor. 
Under these assumptions, $X$ is a \emph{Hadamard manifold}, and is diffeomorphic to a Euclidean space.

\medskip
A semisimple Lie algebra $\G$ is called \emph{compact} if its Killing form is negative definite. A semisimple Lie group $G$ is {compact} if and only if its Lie algebra is compact. $G$ is said to have \emph{no compact factors} or of \emph{noncompact type} if none of the factors of the direct sum decomposition of its Lie algebra $\G$ into simple Lie algebras is compact, and the decomposition has no commutative factors.

\medskip
Let $G$ be a semisimple Lie group with no compact factors and with a finite center, let $K$ be a maximal compact subgroup of $G$. Maximal compact subgroups of $G$ are conjugate to each other. The coset space $X = G/K$ can be given a natural $G$-invariant Riemannian metric with respect to which it becomes a symmetric space of noncompact type. Moreover, under our assumptions, 
$G$ is \emph{commensurable} with the isometry group of $X$, $\isom{X}$, in the sense that the homomorphism 
$G\rightarrow \isom{X}$ has finite kernel and cokernel.
The group $G$ acts on $X = G/K$ transitively, so $X$ is a homogeneous $G$-space.

In fact, any symmetric space of noncompact type arises as a quotient space as above. Let $X$ be a symmetric space of noncompact type, and let $\isomid{X}$ be the identity component of $\isom{X}$. Then $\isomid{X}$ is a semisimple Lie group with no compact factors and  trivial center. We can identify $X$ with the quotient $\isomid{X}/\isomid{X}_x$ where $\isomid{X}_x$ is the isotropy subgroup for some $x\in X$. 

\medskip
In the rest of this paper we reserve the letter $X$ to denote a symmetric space of noncompact type. We identify $X$ with $G/K$ where $G$ \
and $K$ are as above. 
More assumptions on $G$ will be made later on, see section \ref{sec:boundary}.

\medskip
A \emph{flat} in $X$ is a totally geodesic submanifold of zero sectional curvature. A flat is called \emph{maximal} if it is not properly contained in another flat. The group $G$ acts transitively on the set of all maximal flats; the dimension of a maximal flat is called the \emph{rank} of $X$.

A choice of a maximal flat will be called the \emph{model flat}, and will be denoted by $\fmod$. $\fmod$ is isometric to $\mathbb{E}^k$, where $k$ is the rank of $X$. The image of the subgroup $G_\fmod< G$ stabilizing the model flat in the group of isometric affine transformations $\isom{\fmod}$ under restriction homomorphism is a semidirect product $\mathbb{R}^k\rtimes W$. Here $W$, called the \emph{Weyl group}, is a (finite) group of isometries of $\fmod$ generated by reflections fixing a chosen base point (origin) $o_\m$. A fundamental domain for the action $W \curvearrowright \fmod$ is a certain convex cone with tip at $o_\m$, called the \emph{model Weyl chamber}, and will be denoted by $\Delta$.

\subsection{$\Delta$-valued distances}

Given any two points $x,y\in X$, the unique oriented geodesic segment from $x$ to $y$ will be denoted by $xy$. All geodesics considered in this paper are unit speed parametrized. We denote the distance between $x$ and $y$ by $d(x,y)$.

Each oriented segment $xy$ uniquely defines a vector $v$ in $\Delta$ which can be realized as follows.

Any geodesic segment $xy$ can be extended to a complete geodesic $f\subset X$ which is, in fact, a flat of dimension one. This geodesic $f$ is contained in a maximal flat $F$. There exists an isometry $g\in G$ sending $F$ to $\fmod$, $x$ to $o_\m$ and $y$ to $\Delta$. The vector $v\in \Delta$ is defined as the image $g(y)$; it is independent of the choice of $g$.  
This vector $v$ is called the \emph{$\Delta$-valued distance} from $x$ to $y$, and denoted by $d_\Delta(x,y)$. 

\medskip
It follows from our discussion that $d_\Delta(x,y)$ is a complete $G$-congruence invariant for an oriented segment $xy$ or an ordered pair $(x,y)$. Precisely, for two pairs of points $(x,y)$ and $(x',y')$, there exists $g\in G$ satisfying $(gx,gy) = (x',y')$ if and only if $d_\Delta(x,y) = d_\Delta(x',y').$

\medskip
The $\Delta$-valued distances satisfy a set of inequalities which are generalizations of the ordinary triangle inequality (see \cite{kapovich2009convex}). For our purpose, the following form will be sufficient.

\begin{gte}
For any triple of points $x,y,z \in X$,
\[ \| d_\Delta (x,y) - d_\Delta (x,z) \| \le d(y,z), \]
where $d_\Delta (x,y) - d_\Delta (x,z)$ is realized as a vector in $\fmod$ and $\| \cdot \|$ is the induced Euclidean norm.
\end{gte}

\subsection{Ideal boundaries and Tits buildings}\label{sec:boundary}

Two geodesic rays in $X$ are said to be \emph{asymptotic} if they are within a finite Hausdorff distance from each other. The \emph{ideal} or \emph{visual boundary} $\di X$ is the set of asymptotic classes of rays. Given $x \in X$ and an asymptotic class $\zeta$, the unique ray emanating from $x$ which is a member of the asymptotic class $\zeta$ will be denoted by $x\zeta$. For a fixed base point $x\in X$, the set $\di X$ can be metrized by the \emph{angle metric} $\angle_x$,
\[ \angle_x(\zeta_1,\zeta_2) = \text{angle between the rays } x\zeta_1 \text{ and } x\zeta_2.\]
The \emph{visual topology}  on $\di X$ induced by an angle metric $\angle_x$ is independent of the choice of a base point. In fact, $\di X$ is homeomorphic to $S^{n-1}$ where $n$ is the dimension of $X$.

\medskip
The natural \emph{Tits metric} on the ideal boundary $\di X$ can be defined as
\[ \titsangle (\zeta,\eta) = \sup_{x\in X} \angle_x(\zeta,\eta).\]
This metric defines \emph{Tits topology} on $\di X$ which is finer than the visual topology, and $\di X$ equipped with this topology is called the \emph{Tits boundary} of $X$ denoted by $\tits X$.

The Weyl group $W$ acts as a reflection group  
on the Tits boundary $\amod = \tits\fmod \cong S^{k-1}$,  
where $k$ is the rank of $X$. The pair $(\amod, W)$ is called the \emph{spherical Coxeter complex} associated with $X$. The quotient $\smod = \amod/W$ is called the \emph{spherical model Weyl chamber} which we identify as a fundamental chamber of $(\amod, W)$. Accordingly, we regard the model Weyl chamber $\Delta$ of $\fmod$ as a cone in $\fmod$ with tip at the origin $o_\m$ and ideal boundary $\smod$.

\medskip
The spherical Coxeter complex structure on $\amod$ induces a $G$-invariant spherical simplicial structure on $\tits X$. This simplicial complex, called the \emph{spherical} or \emph{Tits building} associated to $X$; we assume this building to be {\em thick}. 
The facets of this simplicial complex are called {\em chambers} in $\tits X$ and the ideal boundaries of maximal flats are \emph{apartments} 
 in $\tits X$.

Each chamber is naturally identified with the model chamber $\smod$ under the projection map (also called the \emph{type map})
\[ \theta : \tits X \rightarrow  \smod.\]
The type map is equivariant with respect to the isometric actions of $\isom{X}$ on $\tits X$ and $\smod$; hence, $G$ acts on $\smod$.  

\medskip 
 {\em From now on, we always assume that $G$ acts on the model chamber $\smod$ trivially}. In particular, $G$ preserves each de Rham factor of $X$ and the type map $\theta$ amounts to the quotient map $\tits X\to \tits X/G$.  

\medskip 
For an ideal point $\zeta$, $\zeta_\m =\theta({\zeta}) \in \smod$ is called the \emph{type} of $\zeta$.
Accordingly, for a face $\tau$ of a chamber $\sigma$, the face $\tmod = \theta({\tau})$ of $\smod$ is called the \emph{type} of $\tau$.

We denote the \emph{opposition involution} on $\smod$ by
\[\iota = -w_0,\]
where $w_0$ denotes the longest element in $W\curvearrowright \amod$.

Two simplices $\tau_1,\tau_2$ of $\tits X$ are called \emph{antipodal} or \emph{opposite} if there exists a point $x\in X$ such that $s_x(\tau_1) = \tau_2$, where $s_x$ is the Cartan involution with respect to $x$. Equivalently, two such simplices are contained in an apartment $a$ such that the antipodal map $-\id$ (induced by a Cartan involution) sends $\t_1$ to $\t_2$. Their types are related by $\theta(\t_1) = \iota \theta(\t_2)$.

\medskip
Throughout the paper, we will consider only $\iota$-invariant faces $\tmod$ of $\smod$. For every such face we pick one and for all 
a fixed point $\xi=\xi_\m$ of $\iota$ in the interior of $\tmod$. Then, for every simplex $\tau$ in $\tits X$ of type $\tmod$, we define 
a point $\xi_\tau\in \tau$ by 
$$
\{\xi_\tau\}= \theta^{-1}(\xi_\m)\cap \tau. 
$$

For a type (face) $\tmod$ of $\smod$ and a point $x\in X$, we define the \emph{$\xi$-angle} between two simplices $\t_1$ and $\t_2$ of type $\tmod$ with respect to $x$ by
\[ \angle^\xi_x (\t_1,\t_2) = \angle_x(\xi_{\t_1},\xi_{\t_2}).\]
Similarly, given $\tmod$-regular segments $xy_1, xy_2$ in $X$, we define the $\xi$-angle 
$$
 \angle^\xi_x (y_1, y_2):= \angle^\xi_x(\t_1,\t_2), 
$$
where $y_i\in V(x, \st{\t_i}), i=1, 2$.

\medskip 
The angular distance $\angle^\xi_x$ induces a \emph{visual topology} on the space of simplices of type $\tmod$. 
The group $G$ acts transitively on this space. The stabilizers $P_\tau$ of simplices $\tau \subset \tits X$ are called the \emph{parabolic subgroups} of $G$. After identifying  $\tmod$ with a simplex $\tau$ of type $\tmod$, the space of simplices of type $\tmod$,
\[ \flag{\tmod} = G/P_{\tmod},\]
is called the \emph{partial flag manifold}.
The topology of $\flag{\tmod}$ as a homogeneous $G$-space agrees with the visual topology.

\subsection{Parallel sets, regularity, cones and diamonds}\label{sec:parallelset}

We often denote a pair of antipodal simplices by $\t_+$ and $\t_-$. Let $\t_\pm$ be a pair of antipodal simplices of same type $\tmod$. Every such pair $\t_\pm$ is contained in a unique minimal\footnote{``Minimal'' means that the dimension of $S$ matches with the dimension of the cells $\t_\pm$.}
singular sphere $S \subset \di X$. The \emph{parallel set} of the pair $\t_\pm$ is defined to be the union of all flats in $X$ which are asymptotic to $S$, and denoted by $\pt$. Equivalently, $\pt$ is the union of all maximal flats $F$ whose ideal boundaries $\di F$ contain $\tau_\pm$. The parallel set $\pt$ is a nonpositively curved symmetric space with Euclidean de Rham factor.

\medskip
In the simplicial complex  $\tits X$, we define the \emph{star} $\st{\t}$, the \emph{open star} $\ost{\t}$ and the \emph{boundary} $\partial\st{\t}$ for a simplex $\t\in \tits X$ as
\[ \st{\t} = \text{minimal subcomplex of $\tits X$ consisting of simplices $\sigma \supset \t$,}\]
\[ \ost{\t} = \text{union of all open simplices whose closure intersects $\Int{\t}$},\]
\[ \partial\st{\t} = \st{\t} - \ost{\t}.\]
Accordingly, we denote the open star and boundary of the star of a model face $\tmod$ in the simplicial complex $\smod$ by   $\ost{\tmod}$ and $\partial\st{\tmod}$, respectively. Note that the simplicial map $\theta: \tits X \rightarrow \smod$ sends $\ost{\t}$ and $\partial\st{\t}$ to $\ost{\tmod}\subset \smod$ and $\partial\st{\tmod}= \smod - \ost{\tmod}$, 
respectively, where $\tmod$ is the type of $\t$.

For a subset $\T \subset\ost{\tmod}$, we define the \emph{$\tmod$-boundary} $\partial\T$ in the topological sense as a subset of $\ost{\tmod}$, where the topology is provided by the Tits metric. We define the interior $\Int{\T}$ of $\T$ as $\T - \partial \T$. If $\Theta$ is compact, then $\epsilon_0({\T}) := \titsangle({\partial\st{\t},\T})  > 0$. Moreover, if $\T'$ and $\T$ are two compact subsets of $\ost{\tmod}$ such that $\T \subset \Int{\T'}$, a scenario we will often consider in our paper, then $ \epsilon_0({\T},\T') := \titsangle(\T,\partial \T')>0$.

A subset $\T$ of $\smod$ is called \emph{$\tmod$-Weyl-convex} if its symmetrization $W_\tmod \T$ in $\amod$ is convex. Here $W_\tmod$ denotes the stabilizer of the face $\tmod$ for the action $W\curvearrowright \amod$. 
For  a ($\tmod$-)Weyl-convex subset $\T \subset \ost{\tmod}$, 
we define the $\T$-star of a simplex $\t \in\tits X$ as
\[  \tst{\t} = \theta^{-1}(\T) \cap \st{\t}.\]
The star $\st{\t}$ and $\T$-stars $\tst{\t}$ of a simplex $\t$ are convex subsets of $\tits X$ with respect to the Tits metric (see \cite{Kapovich:2014aa, MR3736790}). 

\medskip
Define the \emph{$\tmod$-regular} part of the ideal boundary as $\di^{\tmod\mathrm{-reg}} X = \theta^{-1} \ost{\tmod}$. 
An ideal point $\xi$ is called \emph{$\tmod$-regular} if $\xi \in \di^{\tmod\mathrm{-reg}}$.
Given $x\in X$ and $\xi\in\di X$, the geodesic ray $x\xi$ is called $\tmod$-regular if $\xi \in\ost{\tmod}$. A geodesic segment $xy$ is called \emph{$\tmod$-regular} if $xy$ can be extended to a $\tmod$-regular ray $x\xi$. For a Weyl-convex subset $\T\subset\ost{\tmod}$, in a similar fashion we define \emph{$\T$-regularities} for ideal points, rays and segments. Note that a segment $xy$ is $\tmod$-regular if and only if $yx$ is $\iota(\tmod)$-regular.

\medskip
Let $\tmod$ be an $\iota$-invariant face of $\smod$, and $\T$ is an \iinv\ of $\ost{\tmod}$.
Given a point $x\in X$ and a simplex $\t$ of type $\tmod$, the \emph{$\T$-cone} $V(x,\tst{\t})$ with tip $x$ is defined as the union of all $\T$-regular rays $x\xi$ asymptotic to $\st{\t}$. For a $\T$-regular segment $xy$, the $\T$-diamond $\td{x,y}$ is defined as
\[\td{x,y} = V(x,\tst{\t_+}) \cap V(y,\tst{\t_-})\subset \pt,\]
where $\t_\pm$ are unique (unless $x=y$) pair of antipodal simplices in $\flag{\tmod}$ such that $y\in V(x,\tst{\t_+})$ and $x\in V(y,\tst{\t_+})$. The cones and diamonds are convex subsets of $X$, see \cite{Kapovich:2014aa, MR3736790}.

\section{Visual angle estimates}\label{estimates}

The key result in this section is Proposition \ref{tool} which will be used in the proof of Theorem \ref{main_result} to construct \emph{Morse quasigeodesics} (see Definition \ref{MQG_def}). In the first section, we first obtain some weaker results which would lead to the estimates in Proposition \ref{tool} in the later section.

\medskip
In what follows, we always denote by $\tmod$ an $\iota$-invariant face of the model chamber $\smod$. The sets denoted by $\T,\T'$ etc. will always be \iinv\ of $\ost{\tmod}$. By $\xi_\m$ we denote an $\iota$-invariant point in the interior of $\tmod$.

\subsection{Small visual angles I}\label{estimates1}

Define the \emph{space of opposite simplices}
\[\X = \antiflag \underset{\mathrm{open}}{\subset} \flag{\tmod} \times \flag{\tmod},\]
which consists of all pairs of opposite simplices of $\flag{\tmod}$. This space has a transitive $G$-action which makes it a homogeneous $G$-space. The point stabilizer $H$ of this action is the intersection of two opposite parabolic subgroups of $G$.

\medskip
Throughout in this section $x$ will be a fixed point of $X$. For a point $\omega = (\tau_+,\tau_-) \in \X$, let $P(\omega)$ denote the parallel set $P(\tau_+,\tau_-)$. We define a function $\dopp_x : \X \rightarrow \rplus$ by
\[\dopp_x\left(\omega\right) = d\left(x, P(\omega)\right).\]

\begin{prop}\label{dopp}
The function $\dopp_x$ is continuous.
\end{prop}

\begin{proof} The proof is the same as of Lemma 2.21 of \cite{ MR3736790}. 
Fix a point $\omega_0\in \X$. From the fiber bundle theory, we have a fibration 
\[ H \rightarrow G \xrightarrow{~\ev~} \X,\]
where $H$ denotes the point stabilizer of the transitive $G$ action, and $\ev$ denotes the evaluation map $\ev(g) = g\cdot \omega_0$. See \cite[Sections 7.4, 7.5]{MR1688579}. For any $\omega \in \X$, there exists a neighborhood $U$ such that $\ev$ has a local section $\sigma$ over $U$,
\[ \sigma: U \rightarrow G, \quad \ev\circ s = \id_U. \]
It suffices to show that $\dopp_x$ is continuous on such neighborhoods $U$.

Define a function $d': X\times\X \rightarrow \rplus$ by $d'(x,\omega) = \dopp_x(\omega)$.
Note that the action of $G$ on $X\times \X$ given by $g(x,\omega) = (gx,g\omega)$ leaves $d'$ invariant. Therefore, on $U$,
\[ \dopp_x(\omega) = d'(x,\omega) = d'(x,\sigma(\omega) \omega_0) = d'(\sigma(\omega)^{-1}x, \omega_0) = d(s(\omega)^{-1}x, P(\omega_0)), \]
where the last function is continuous on $U$. Therefore, $\dopp_x$ is continuous on $U$.
\end{proof}

\begin{defn}[Antipodal subsets]
 A pair of subsets $\Lambda_1$, $\Lambda_2$  of $\flag{\tmod}$ is called \emph{antipodal}, if any simplex $\tau_1\in \Lambda_1$ is antipodal\footnote{See subsection \ref{sec:boundary} for the definition of antipodal simplices.} to any simplex $\tau_2\in \Lambda_2$ and vice versa. 
\end{defn}

Let $\Lambda_1$ and $\Lambda_2$ be a pair of compact, antipodal subsets of $\flag{\tmod}$. Then, $\Lambda_1\times\Lambda_2$ is a compact subset of $\X$.

Proposition \ref{dopp} implies:

\begin{cor}\label{D}
Let $\Lambda_1$ and $\Lambda_2$ be compact, antipodal subsets of $\flag{\tmod}$. If $\Lambda_1$ and $\Lambda_2$ are antipodal, then, for any point $x\in X$, there is a number $D = D(\Lambda_1,\Lambda_2,x)$ such that
\[d(x, P(\tau_1,\tau_2)) \leq D, \quad \forall\tau_1\in \Lambda_1, \forall\tau_2\in \Lambda_2.\]
\end{cor}

\begin{prop} \label{f}
Let $\Lambda_1,\Lambda_2 \subset \flag{\tmod}$ be compact, antipodal subsets. There exists a function 
$f = f(\Lambda_1,\Lambda_2,x):[0,\infty) \rightarrow [0,\pi]$ 
satisfying $f(R) \rightarrow 0$ as $R\rightarrow \infty$ such that 
for any $\tau_1\in \Lambda_1$, $\tau_2\in \Lambda_2$, and 
for any $z_1\in x\xi_{\tau_1}$, $z_2\in x\xi_{\tau_2}$ satisfying $d(z_1,x), d(z_2,x) \ge R$, 
we have
\[ 
\alpha_1 = \angle_{z_1} (x,z_2)\leq f(R), \quad 
\alpha_2 = \angle_{z_2} (x,z_1) \leq f(R).
\]
\end{prop}

\begin{proof}
Let $\bar{x}\in P(\tau_1,\tau_2)$ be the point closest to $x$. Since $\tau_1$ and $\tau_2$ are antipodal, $\angle^\xi_{\bar{x}}(\tau_1,\tau_2) = \pi$, i.e. $s_{\bar{x}} (\xi_{\tau_1}) = \xi_{\tau_2}$. Let $c : (-\infty,\infty) \rightarrow P(\tau_1,\tau_2)$, $c(0) = \bar{x}$, be the biinfinite geodesic passing through $\bar{x}$ and asymptotic to $c(+\infty) = \xi_{\tau_1}$ and $c(-\infty) = \xi_{\tau_2}$. For $i=1,2$, let $c_i : [0,\infty) \rightarrow X$ be the geodesic ray  $x\xi_{\tau_i}$ (see Figure \ref{fig:f}(a)). Since the functions $d(c(t), c_1(t))$ and $d(c(-t), c_2(t))$ are bounded convex functions, they are decreasing with maximum at $t=0$. Therefore,
\begin{equation}\label{H_dist}
d(c(t), c_1(t))\leq D, ~d(c(-t), c_2(t)) \leq D, \quad\forall t\in[0,\infty),
\end{equation}
where $D>0$ is a number as in Corollary \ref{D}.

\begin{figure}[h]
\begin{center}
\begin{tikzpicture}
    \node[anchor=south west,inner sep=0] (image) at (0,0,0) {\includegraphics[width=5in]{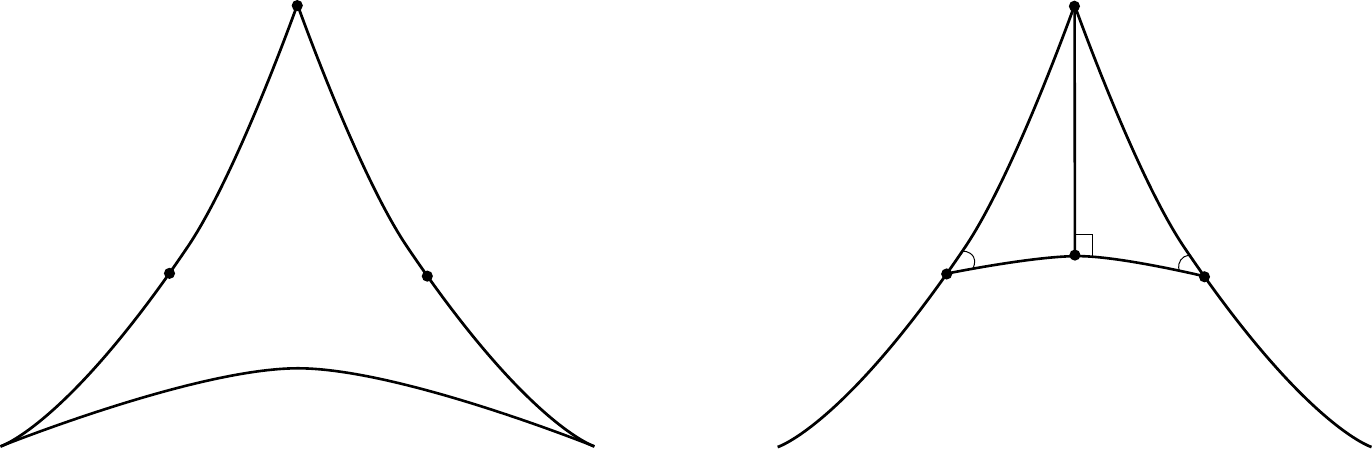}};
    \begin{scope}[x={(image.south east)},y={(image.north west)}]
        \node at (.234,1) {$x$};
        \node at (.799,1) {$x$};
        \node at (.1,.394) {$z_2$};
        \node at (.335,.389) {$z_1$};
        \node at (.663,.394) {$z_2$};
        \node at (.902,.389) {$z_1$};
        \node at (.785,.373) {$x'$};
        \node at (0.225,.127) {$\bar{x}$};
        \node at (.839,.45) {$\scriptstyle\alpha_1$};
        \node at (.728,.45) {$\scriptstyle\alpha_2$};
        \node at (.81,.57) {$\scriptstyle\leq 2D$};
        \node at (.841,.73) {$\scriptstyle\geq R$};
        \node at (.725,.73) {$\scriptstyle R \leq$};
        \node at (.55,0) {$\xi_{\tau_2}$};
        \node at (1.03,0) {$\xi_{\tau_1}$};
        \node at (-.015,0) {$\xi_{\tau_2}$};
        \node at (.46,0) {$\xi_{\tau_1}$};
        \node at (.22,-.19) {(a)};
        \node at (.78,-.19) {(b)};
        \node at (.149,.74) {$\scriptstyle c_2(t)$};
        \node at (.286,.74) {$\scriptstyle c_1(t)$};
        \node at (.3,.074) {$\scriptstyle c(t)$};
        \filldraw (0.225,.18) circle (1.37pt);
    \end{scope}
\end{tikzpicture}
\end{center}
\captionsetup{width=0.7\textwidth}
\caption{}
\label{fig:f}
\end{figure}

For $R\ge 0$, let $c_1(t_1) = z_1\in x\xi_{\tau_1}$ and $c_2(t_2) = z_2\in x\xi_{\tau_2}$ be any points satisfying $t_1 = d(z_1,x)\geq R$ and $t_2 = d(z_2,x) \ge R$. By (\ref{H_dist}), the Hausdorff distance between the segments $z_1z_2$ and $c(t_1)c(-t_2)$ is bounded above by $D$. Combining with $d(x,\bar{x}) \leq D$ we obtain
\begin{equation}\label{H_haus}
d(x, z_1z_2) \le 2D.
\end{equation}

Let $x'$ be the point on $z_1z_2$ nearest to $x$. When $R\ge 2D+1$, $x'$ is in the interior of $z_1z_2$. 
Consider geodesic triangles $\triangle_1 = \triangle (x,x',z_1)$ and $\triangle_2 = \triangle (x,x',z_2)$; 
the angle of $\triangle_1$ and $\triangle_2$ at the vertex $x'$ is $\pi/2$. 
Let $\alpha_1 =\angle_{z_1} (x,x')= \angle_{z_1} (x,z_2)$ and $\alpha_2 =\angle_{z_2} (x,x')= \angle_{z_2} (x,z_1)$ (see Figure \ref{fig:f}(b)). Let $\tilde{\triangle}_1$ and $\tilde{\triangle}_2$ be the Euclidean comparison triangles of $\triangle_1$ and $\triangle_2$, respectively; we denote the corresponding vertices of $\tilde{\triangle}_1$ and $\tilde{\triangle}_2$ by the same symbols. In the triangles $\tilde{\triangle}_1$, $\tilde{\triangle}_2$, since the angles at the vertex $x'$ are at least $\pi/2$, we have
\[
\tilde{\alpha}_i \leq \sin^{-1} \left( \frac{xx'}{xz_i}\right) 
\leq \sin^{-1} \left( \frac{2D}{R}\right),
\qquad i=1,2,
\]
where $\tilde{\alpha}_i$ denotes the angle corresponding to $\alpha_i$. The second inequality in above comes from (\ref{H_haus}). Since the triangles $\triangle_1$ and $\triangle_2$ are thinner than the triangles $\tilde{\triangle}_1$ and $\tilde{\triangle}_2$, respectively, we have $\alpha_i \leq \tilde{\alpha}_i$. Therefore, when $R\ge 2D+1$, $f(R)$ can be given by the following formula:
\[
f(R) = \sin^{-1} \left( \frac{2D}{R}\right).
\]

The domain of $f$ can be extended to $R< 2D+1$ continuously. However, the continuity of $f$ is irrelevant; we can simply set $f(R) = \pi$ for $R< 2D+1$.
\end{proof}

We also give a $\xi$-angle version of the proposition above which will be useful in the next section.

\begin{prop}\label{small_angles_preffered}
Let $\Lambda_1,\Lambda_2 \subset \flag{\tmod}$ be compact antipodal subsets. Given $\Theta\subset \ost{\tmod}$ containing $\xi_\m$ in its interior, there exists $R_0 = R_0(x,\Lambda_1,\Lambda_2,\Theta,\xi)$ such that 
for any $\tau_1\in \Lambda_1$, $\tau_2\in \Lambda_2$, and for any $z_1\in x\xi_{\tau_1}$, $z_2\in x\xi_{\tau_2}$ satisfying $d(z_1,x), d(z_2,x) \ge R_0$, the segment $z_1z_2$ is $\Theta$-regular.

Moreover, there exists a function 
$f_0 = f_0(x,\Lambda_1,\Lambda_2,\xi):[0,\infty) \rightarrow [0,\pi]$ 
satisfying $f_0(R) \rightarrow 0$ as $R\rightarrow \infty$ such that for any $\tau_1\in \Lambda_1$, $\tau_2\in \Lambda_2$, and for any $z_1\in x\xi_{\tau_1}$, $z_2\in x\xi_{\tau_2}$ satisfying $d(z_1,x), d(z_2,x) \ge R \ge R_0$, we have
\[ \angle^\xi_{z_1} (x,z_2), \angle^\xi_{z_2} (x,z_1) \leq f_0(R).\]
\end{prop}

\begin{proof}
Let $\alpha = \min\left\{ \titsangle(\xi,\zeta)\mid \zeta\in\partial\Theta\right\}>0$. Using the \tedl, we get
\[
\left\| d_\Delta(x,z_1) - d_\Delta(x_1,z_1) \right\| \leq d(x,x_1),
\]
for any point $x_1\in X$. Specializing to $x_1 = x'$, the point on $z_1z_2$ closest to $x$, we obtain
\[ 
\left\| d_\Delta(x,z_1) - d_\Delta(x',z_1) \right\| \leq 2D.
\]
Then $x'z_1$ is $\Theta$-regular when $xz_1$ has length $\geq 2D/\sin\alpha$. Therefore, the constant $R_0$ can be given by
\begin{equation}\label{R_0}
R_0 = \frac{2D}{\sin\alpha}.
\end{equation}
This proves first part of the proposition.

\medskip
For the second part, let $\left(\Theta_n\right)_{n\in\mathbb{N}}$ be a nested sequence of $\iota$-invariant, Weyl-convex, compact subsets of $\ost{\tmod}$ such that $\xi$ is an interior point of each $\Theta_n$, and $\bigcap_{n=1}^\infty \Theta_n = \{\xi\}$. Let $\alpha_n$ be the Tits-distance from $\xi$ to the boundary of $\Theta_n$,
\[\alpha_n = \min\left\{ \titsangle(\xi,\zeta)\mid \zeta\in\partial\Theta_n\right\}>0.\]
Clearly, $\left(\alpha_n\right)_{n\in\mathbb{N}}$ is a strictly decreasing sequence converging to zero. This implies that $R_0(\Theta_n)$ is strictly increasing which diverges to infinity, where $R_0$ is as in (\ref{R_0}).
If $R_0(\Theta_n) \le d(x,z_1) < R_0(\Theta_{n+1})$, then the first part of the proposition  implies that $z_2z_1$ is $\Theta_n$-regular, which then implies
\begin{align*}
\angle^\xi_{z_2} (x,z_1)
&\le \angle_{z_2}(x,z_1) + \angle_{o_{\mathrm{mod}}} (\xi,d_\Delta(z_2,z_1))\\
&\le f(R) + \alpha_n,
\end{align*}
where the function $f$ is as in Proposition \ref{f}.
Therefore, when $R_0(\Theta_n) \le R < R_0(\Theta_{n+1})$, we may define 
\begin{equation*}
f_0(R) = f(R) + \alpha_n.
\end{equation*}
 As in the case of $f$ in Proposition \ref{f}, continuity of $f_0$ is irrelevant.
\end{proof}

\subsection{Small visual angles II}\label{estimates2}

The $\Theta$-cones (over a fixed simplex $\tau\in\flag{\tmod}$) vary continuously with their tips. Here, the topology on the set of $\Theta$-cones over a fixed simplex $\tau$ is given by their Hausdorff distances. Precisely, we have,

\begin{thm}[Uniform continuity of $\Theta$-cones, \cite{Kapovich:2014aa}]\label{hd}
The Hausdorff distance between two $\Theta$-cones over a fixed $\tau\in\flag{\tmod}$ is bounded by the distance between their tips,
\[\hausdist{V(x,\tst{\tau})}{V(\bar{x},\tst{\tau})} \leq d(x,\bar{x}).\]
\end{thm}

Moreover, for diamonds, one also has the following form of uniform continuity. This will be useful in our paper, especially in the discussion of replacements (section \ref{sec:replacements}).

\begin{thm}[Uniform continuity of diamonds]\label{cd}
Given any $\Theta'$ with $\Int{\Theta'} \supset \Theta$, and any $\delta  > 0$, there exists $c = c(\Theta,\Theta',\delta)$ such that for all $\Theta$-regular segments $xy$ and $x'y'$ with $d(x,x') \le \delta$, $d(y,y') \le \delta$, we have
\[ \td{x,y} \subset \nbhd{c}{\tdd{x',y'}}.\]
\end{thm}
\begin{proof}
   We will prove this theorem as a corollary of  
\cite[Theorem 5.16]{Kapovich:2014ab}: For every $(\Theta,B)$-regular 
$(L,A)$-quasigeodesic $q: [a_-, a_+]\to X$ and points $x_\pm\in X$ within distance $\le B$ from  $q(a_\pm)$, 
the image of $q$ is contained in the $D(L,A,\Theta, B)$-neighborhood of  the diamond $\tmd{x_-, x_+}$.  

\begin{rem}
Using the hard theorem \cite[Theorem 5.16]{Kapovich:2014ab} in order to prove 
Theorem \ref{cd} is an overkill, but it is quicker than a direct argument.  We refer the reader to section \ref{sec:morse} for the definition 
of $(\Theta,B)$-regular quasigeodesics. 
\end{rem}
 
By appealing to the triangle inequalities for $\Delta$-length, one gets a slightly more precise statement, namely, there exists 
$D(L,A,\Theta, \Theta', B)$ such that the image of $q$ is contained in the 
$D(L,A,\Theta, \Theta', B)$-neighborhood of the diamond $\tdd{x_-, x_+}$. 

We observe that for every point $z\in \td{x, y}$ the broken geodesic segment 
$$
xz \star zy
$$
is $(L,0)$-quasigeodesic for some $L=L(\Theta)$,  and is $(\Theta, 0)$-regular. Hence, according to
the above sharpening of  \cite[Theorem 5.16]{Kapovich:2014ab}, 
the point $z$ belongs to the $c=D(L,0,\Theta,\Theta',\delta)$-neighborhood of the diamond $\tdd{x', y'}$, provided that
$$
d(x, x')\le \delta, d(y, y')\le \delta. 
$$
Thus,
\[
 \td{x,y} \subset \nbhd{c}{\tdd{x',y'}}.\qedhere 
\]
\end{proof} 

Now we turn to the main estimate in this section.

\begin{prop}[Uniformly small visual angles]\label{tool}
Let $\Lambda_1,\Lambda_2\subset\flag{\tmod}$ be compact, antipodal sets, and let $\Theta'$ be a subset of $\ost{\tmod}$ containing $\Theta$ in its interior.
Let $y_1 \in V(x,\tst{\tau_1})$ and $y_2\in V(x,\tst{\tau_2})$ be any points, where $\tau_1\in\Lambda_1$ and $\tau_2\in \Lambda_2$ are any simplices. Then,
\begin{enumerate}

\item \label{tool2}  There exists a constant $R_1 = R_1(x,\Lambda_1,\Lambda_2,\Theta',\Theta)$ such that $y_1y_2$ is $\Theta'$-regular if $d(x,y_i) \ge R_1$.

\item \label{tool3} There exists a function $f_1 = f_1(x,\Lambda_1,\Lambda_2,\Theta',\Theta,\xi): [0,\infty) \rightarrow [0,\pi]$ satisfying $\lim_{R\rightarrow \infty} f_1(R) = 0$ such that if $d(x,y_i) \ge R\ge R_1$, then
\begin{equation}\label{f_1}
\angle_{y_1}^\xi(x,y_2), \angle_{y_2}^\xi(x,y_1) \leq f_1(R).
\end{equation}

\end{enumerate}
\end{prop}

\begin{proof}
For part \ref{tool2}, we take an approach similar to the one given in the proof of Proposition \ref{small_angles_preffered}. Let $\bar{x}$ be the nearest point projection of $x$ into the parallel set $P(\tau_1,\tau_2)$, and for each $i=1,2$ let $\bar{y}_i$ denote the nearest point projection of $y_i$ into $V(\bar{x},\tst{\tau_i}) \subset P(\tau_1,\tau_2)$.
Let $\alpha =  \titsangle(\Theta,\partial\Theta') >0$, and $\alpha' = \titsangle(\Theta,\partial\st{\tmod}) \ge \alpha$. Finally, let $D = D(\Lambda_1,\Lambda_2,x)$ be the constant given by Corollary \ref{D}. 

Since $d(x,\bar{x}) \le D$, we combine this with Theorem \ref{hd} to get
\[ d(y_i,\bar{y}_i) \le D, \quad i=1,2.\]
Then, using the triangle inequality for $\Delta$-lengths we deduce
\begin{multline*}
\left\| d_\Delta(y_1,y_2) - d_\Delta(\bar{y}_1, \bar{y}_2) \right\| \leq \left\| d_\Delta(y_1,y_2) - d_\Delta(\bar{y}_1,y_2) \right\| + \left\| d_\Delta(\bar{y}_1,y_2) - d_\Delta(\bar{y}_1, \bar{y}_2) \right\| \\
 \le d(y_1,\bar{y}_1) + d(y_2,\bar{y}_2) \le 2D.
\end{multline*}
Since $\bar{y}_1\bar{y}_2$ is $\Theta$-regular, $y_1y_2$ is $\Theta'$-regular whenever $y_1y_2$ has length $\ge 2D/\sin\alpha$.
Moreover,
\begin{multline}\label{4D}
d(y_1,y_2) \ge d(\bar{y}_1,\bar{y}_2)-2D \ge d(\bar{y}_i,\bar{x})\sin\alpha' -2D \ge (d(y_i,x)-2D)\sin\alpha -2D \\
= d(y_i,x)\sin\alpha -2D(1+\sin\alpha),
\end{multline}
where the second inequality comes from triangle comparisons.
 Using (\ref{4D}), we obtain: $d(y_1,y_2) \ge 2D/\sin\alpha$ whenever $d(x,y_1)$ or $d(x,y_2)$ is greater than $2D(1/\sin^2\alpha+1/\sin\alpha+1)$. We may set
\[ R_1 = 2D\left(1+\frac{1}{\sin\alpha} + \frac{1}{\sin^2\alpha}\right).\] 
This proves part \ref{tool2}.

\medskip
To prove part \ref{tool3} we need the following lemmas.

Recall that $s_x : X \r X$ denotes the Cartan involution of $X$ fixing $x$.

\begin{lem}\label{lem1}
Let $\tau,\tau'\in\flag{\tmod}$ be a pair of simplices, let $x\in X$ be any point, and let $y\in V(x,\tst{s_x\tau})$ be a point satisfying  $d(x,y) \ge l$. 
For sufficiently small $\epsilon$, $\epsilon\le \epsilon_0(\xi_\m)$, we have: 
If $\angle_x^\xi(\tau,\tau') \le \epsilon$, then
\[
\angle_y^\xi(\tau,\tau') \le \epsilon'(\Theta,l)
\]
with $\epsilon'(\Theta,l) \rightarrow 0$ as $l\rightarrow \infty$.
\end{lem}

\begin{proof}
This is a restatement of \cite[Lemma 2.44(ii)]{Kapovich:2014aa}.
\end{proof}

In the following, $\Theta''$ will denote an auxiliary subset of $\ost{\tmod}$ such that 
$\Int{\T''} \supset \T$. 
Let $\alpha'' = \titsangle(\Theta,\partial\T'')$.

\begin{lem}\label{lem2}
Let  $\tau \in \flag{\tmod}$ be any simplex, and $y\in V(x,\tst{\tau})$ be any point. If $z \in x\xi_\tau$ is any point such that $d(x,y)\sin\alpha'' \ge d(x,z)$, then
\[
y\in V(z,\tttst{\tau}).
\]
\end{lem}

\begin{proof}
Let $F$ be a maximal flat asymptotic to $\tau$ containing $x$ and $y$, and let $y'$ be the nearest point projection of $y$ into $x\xi_\tau$. Since $\xi \in \tmod$, the Tits distance from $\xi$ to any point in $\Theta$ is bounded above by $\pi/2-\epsilon(\T)$ where $\epsilon(\T)>0$. Then, the distance from $x$ to $y'$ is comparable to $d(x,y)$, i.e.
\[
d(x,y)\cos(\theta) = d(x,y'), \quad \theta = \angle_x(y,y') \le \pi/2 - \epsilon(\Theta).
\]
Notice that $0<\alpha'' \le \theta + \alpha'' \le \pi/2-\epsilon(\Theta'')$. 
For any point $z'\in xy'$, $y\in V(z',\tttst{\tau})$ whenever $d(y,y') \le d(z',y') \tan(\theta + \alpha'')$. 

Let $z' \in xy'$ be a point which satisfies $d(y,y') = d(z',y') \tan(\theta + \alpha'')$. In that case,
\begin{multline*}
 d(x,z') = d(x,y') - d(z',y') = d(x,y)\cos\theta - d(y,y')\cot(\theta+\alpha'') \\
 =d(x,y)\cos\theta - d(x,y)\sin\theta\cot(\theta+\alpha'') = d(x,y)\frac{\sin{\alpha''}}{\sin(\theta+\alpha'')} \ge d(x,y)\sin\alpha''.
\end{multline*}
Moreover, if $z \in x\xi_\tau$ is the point satisfying $d(x,z) = d(x,y)\sin\alpha''$, then $z' \in V(z,\tttst{\tau})$, and from convexity of cones, $y \in V(z,\tttst{\tau})$.
\end{proof}

\begin{lem}\label{lem3}
There exists a function $f_1'(x,\Lambda_1,\Lambda_2,\Theta,\xi) : [0,\infty) \rightarrow [0,\pi]$ satisfying $ f_1'(R) \r 0$ as $R\rightarrow \infty$
 such that the following holds: For $\tau_1\in\Lambda_1$, let $y_1\in V(x,\tst{\tau_1})$ be any point. If $d(x,y_1) \ge R$, then
\begin{equation*}
\max_{\tau_2\in \Lambda_2}\angle_{y_1}^\xi(x,\tau_2) \leq f_1'(R).
\end{equation*}
\end{lem}

\begin{proof}
Using Lemma \ref{lem2}, if $z_1\in x\xi_{\tau_1}$ is satisfies  $d(x,z_1) = d(x,y_1)\sin\alpha''$, then $y_1\in V(z_1,\tttst{\tau_1})$. See Figure \ref{fig:small_angles}. Letting $d(x,z_2)\rightarrow \infty$ in Proposition \ref{small_angles_preffered}, we get
\[ \angle_{z_1}^\xi (s_x(\tau_1),\tau_2) = \angle_{z_1}^\xi (x,\tau_2) \le f_0(R\sin\alpha''), \quad \forall\tau_2\in\Lambda_2.\]
When $R$ is sufficiently large, $R\ge R_2(x,\Lambda_1,\Lambda_2,\xi)$, then $f_0(R\sin\alpha) \le \epsilon_0(\xi_\m)$, where $\epsilon_0(\xi_\m)$ is as in Lemma \ref{lem1}. Moreover, since $d(y_1,z_1)\ge R(1-\sin\alpha'')$, Lemma \ref{lem1} implies that
\[ \angle_{y_1}^\xi (x,\tau_2) = \angle_{y_1}^\xi (s_x(\tau_1),\tau_2) \le \epsilon'(\Theta'', R(1-\sin\alpha'')), \quad\forall \tau_2\in\Lambda_2.\]

\begin{figure}[h]
\begin{center}
\begin{tikzpicture}
    \node[anchor=south west,inner sep=0] (image) at (0,0,0) {\includegraphics{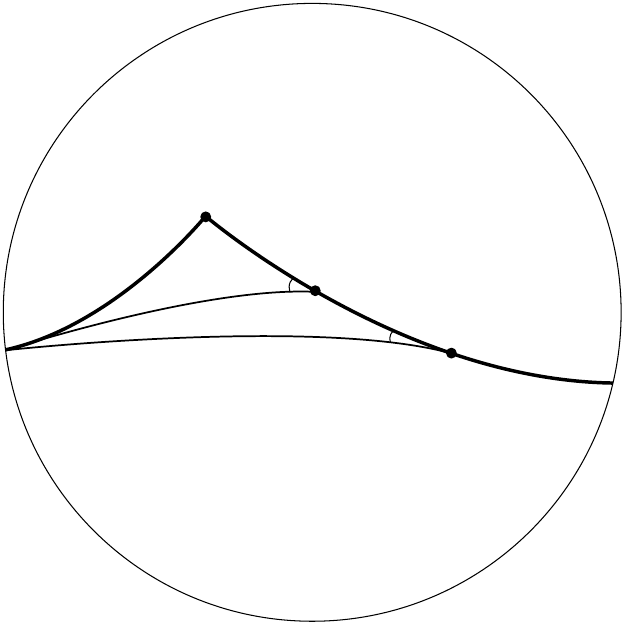}};
    \begin{scope}[x={(image.south east)},y={(image.north west)}]
        \node at (0.33,0.7) {$x$};
        \node at (0.546,0.57) {$z_1$};
        \node at (0.74,0.474) {$y_1$};
        \node at (1.02,0.37) {$\tau_1$};
        \node at (-.02,0.43) {$\tau_2$};
        \node at (0.4,0.56) {$\scriptstyle\epsilon_0 \ge$};
        \node at (0.559,0.4821) {$\scriptstyle\epsilon' \ge$};
    \end{scope}
\end{tikzpicture}
\end{center}
\captionsetup{width=0.7\textwidth}
\caption{A schematic picture depicting small angles. The thick lines are Weyl cones $V(x,\tst{\tau_1})$ and $V(x,\tst{\tau_2})$, while the thin lines are the geodesic rays $z_1\xi_2$ and $y_1\xi_2$. }
\label{fig:small_angles}
\end{figure}
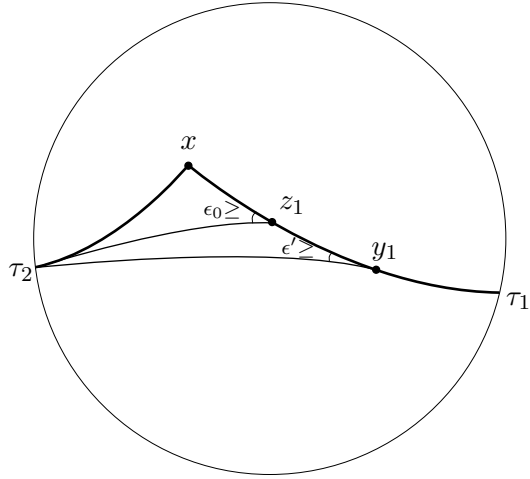

So, we may define
\begin{equation*}
f_1'(R) =
\left\{
\begin{array}{ll}
\epsilon'(\Theta'', R(1-\sin\alpha'')), &\text{if }R\ge R_2\\
\pi, &\text{otherwise}
\end{array}
\right..
\end{equation*}
\end{proof}

Now we are ready to prove the estimate (\ref{f_1}). We first observe that the only property of the point $x\in X$ we have used to estimate the functions $f$ in Proposition \ref{f}, $R_0$ and $f_0$ in Proposition \ref{small_angles_preffered} and subsequently $f_1'$ in Lemma \ref{lem3} is that
\[ d(x, P(\tau_1,\tau_2)) \leq D, \quad \forall\tau_1\in\Lambda_1, \forall\tau_2\in\Lambda_2.\]
All these estimates for $x$ work for any other point $x_1$ satisfying this inequality with the same number $D$.
Moreover, all these estimates work if we replace $\Lambda_1$ or $\Lambda_2$ by their proper subsets. In particular, we may replace $\Lambda_2$ by any of its singleton subsets $\{\tau_2\}$, or replace $x$ by a point $y_2\in V(x,\tst{\tau_2})$. Here we use the fact that for a fixed $\tau_2$ and a point $y_2$ in $V(x,\tst{\tau_2})$,
\[ d(y_2,P(\tau_1,\tau_2)) \le d(x, P(\tau_1,\tau_2)) \leq D, \quad \forall\tau_1\in\Lambda_1.\]
Therefore, the same estimate $f_1'(x,\Lambda_1,\Lambda_2,\xi)$ works when $x$  and $\Lambda_2$ (elsewhere) in Lemma \ref{lem3} are replaced by $y_2$ and $\{\tau_2\}$, respectively. 
Precisely, whenever $y_1y_2$ is a $\Theta'$-regular,
\begin{equation}\label{y_2}
\angle_{y_1}^\xi(y_2,\tau_2) \leq f_1'(d(y_1,y_2)),
\end{equation}
where $f_1' = f_1'(x,\Lambda_1,\Lambda_2,\Theta',\xi)$. $\Theta'$-regularity of $y_1y_2$ is also guaranteed whenever, for $i=1,2$, $d(x,y_i) \ge R_1(x,\Lambda_1,\Lambda_2,\Theta',\Theta)$.

Therefore, if $R\ge R_1(x,\Lambda_1,\Lambda_2,\Theta',\Theta)$ and $d(x,y_1),d(x,y_2)\ge R$, we can use Lemma \ref{lem3}, (\ref{4D}) and (\ref{y_2}) to get
\begin{multline*}
\angle_{y_1}^\xi(x,y_2) \le \angle_{y_1}^\xi(x,\tau_2) + \angle_{y_1}^\xi(y_2,\tau_2)\\
\le f_1'(R) + f_1'(R\sin\alpha -2D(1+\sin\alpha))
\le 2f_1'(R\sin\alpha -4D),
\end{multline*}
where $\alpha =  \titsangle(\Theta,\partial\Theta') $.

This completes the proof of part \ref{tool3}.
\end{proof}

\section{Morse condition}\label{sec:morse}

In this section, we discuss Morse quasigeodesics, Morse embeddings and Morse subgroups and their various properties. These notions were introduced in \cite{Kapovich:2014aa}, and it was proved that the notions of Morse subgroups and Anosov subgroups are equivalent (Theorem \ref{equiv_RCA}). 

\medskip
One of the important new result in this section is the \emph{replacement lemma} (see Theorems \ref{replacements} and \ref{replacements2}) which will be proved in section \ref{sec:replacements}. This will be an important ingredient in the proof of Theorem \ref{main_result}.

\subsection{Stability of quasigeodesics}\label{stability}
Recall that a \emph{quasigeodesic} in a metric space $(Y,d_Y)$ is a coarse-isometric embedding of an interval $I\subset\R$ into $Y$. Quantitively speaking, an $(L,A)$-quasigeodesic in $Y$ is a map, not necessarily continuous, $\gamma : I \rightarrow Y$ which satisfies
\[ L^{-1}|a-b| -A \le d_Y(\g(a),\g(b)) \le L|a-b| + A,\]
where $d_Y$ is the metric of $Y$. When $Y$ is assumed to be a geodesic $\delta$-hyperbolic space, the Morse lemma, proven for these spaces by Gromov \cite[Proposition  7.2.A]{MR919829}, establishes stability of quasigeodesics. Precisely, an $(L,A)$-quasigeodesic in a $\delta$-hyperbolic space stays within a uniform neighborhood of a geodesic; the radius $H$ of this neighborhood solely depends on the given parameters, namely $L,A$ and $\delta$,
\[H = L^2(A_1A + A_2\delta),\]
where $A_1$ and $A_2$ are universal constants, see \cite{MR3003738}. The stability of quasigeodesics can also be stated without referring to geodesics: An $(L,A)$-quasigeodesic path is \emph{stable} if the image of any $(L',A')$-quasigeodesic with the same endpoints is uniformly close to the original path. Thus, any uniform quasigeodesic in a $\delta$-hyperbolic space is stable. Morse lemma is a vital ingredient to prove the invariance of hyperbolicity under quasiisometries, see \cite[Corollary 11.43]{Drutu:2018aa}.

\medskip
One of the major differences between the coarse geometric nature of $\cat0$ (or \emph{non-positively curved}) and $\delta$-hyperbolic  spaces is that the quasigeodesics in $\cat0$ spaces can be unstable, and thus the most naive generalization of the Morse lemma fails in the $\cat0$ settings, already for the Euclidean plane. 
Some versions of the Morse lemma are known for $\cat0$ spaces; in \cite{MR3175245} it has been shown that a quasigeodesic is stable if and only if it is \emph{strongly contracting}. However, this class of quasigeodesics is too restrictive in the context of symmetric spaces.

Nevertheless, according to the main theorem of 
\cite{Kapovich:2014ab}, an analogue of the Morse lemma
holds for  $\tmod$-\emph{regular quasigeodesics}, with 
diamonds (or, cones, or parallel sets) 
replacing geodesic segments (rays, complete geodesics). 

\medskip
The letters $B$, $D$ appear bellow are non-negative numbers.

\begin{defn}[Regular quasigeodesics]
A pair of points in $X$ is called \emph{$\Theta$-regular} if the geodesic segment connecting them is $\Theta$-regular. An $(L,A)$-quasigeodesic $\gamma: I\rightarrow X$ is called \emph{$(\Theta,B)$-regular} if for all $t_1,t_2 \in I$, $|t_1 - t_2| \ge B$ implies that $(\gamma(t_1),\gamma(t_2))$ is $\Theta$-regular.
\end{defn}

In \cite[Theorem 5.17]{Kapovich:2014ab}, it is shown that (finite) regular quasigeodesics are special in the sense that they live very close to the diamonds. We state this result in the next theorem.

\begin{thm}[Morse Lemma for Symmetric Spaces of Higher Rank]\label{morse_lemma} Let $\gamma: [a,b] \rightarrow X$ be a $(\Theta,B)$-regular $(L,A)$-quasigeodesic. There exists a constant $D = D(L,A,\Theta,\Theta',B,X)>0$ such that the image of $\g$ is contained in the $D$-neighborhood of a diamond $\diamondsuit_{\Theta'}(x_1,x_2)$ with tips satisfying $d(\g(a),x_1)\le D$, $d(\g(b),x_2) \le D$.
\end{thm}

Now we review the notion of Morse quasigeodesics.

\begin{defn}[Morse quasigeodesics, \cite{Kapovich:2014aa}]\label{MQG_def}
A (finite, semiinfinite, or biinfinite) $(L,A)$-quasi\-geodesic $\gamma: I\rightarrow X$ is called a \emph{$(L,A,\Theta,D)$-Morse quasigeodesic} if for all $t_1,t_2\in I$, the image $\gamma{([t_1,t_2])}$ is $D$-close to a $\Theta$-diamond $\diamondsuit_\Theta(x_1,x_2)$ with tips satisfying $d(x_i, \gamma(t_i)) \le D$, for $i=1,2$. 
\end{defn}

\begin{rem}~

\begin{enumerate}
\item In light of this definition, the Theorem \ref{morse_lemma} is equivalent to saying that \emph{the uniformly regular uniform quasigeodesics are uniformly Morse}. Conversely, uniformly Morse quasigeodesics are also uniformly regular.
\item Note that it is not in general true that for an $(L,A,\Theta,D)$-Morse quasigeodesic $\g$, the segment $\g(t_1)\g(t_2)$ is $\tmod$-regular. However, when $t_2 -t_1$ is uniformly large, $\g(t_1)\g(t_2)$ becomes uniformly $\tmod$-regular, and in this case one can say that $\g([t_1,t_2])$ lies in a uniform neighborhood of $\tdd{\g(t_1),\g(t_2)}$ for any subset $\T'\subset \ost{\tmod}$ containing $\T$ in its interior (cf. Theorem \ref{cd}).
\end{enumerate}
\end{rem}

\subsection{Straight sequences}

We review some important tools for constructing Morse quasigeodesics.

\medskip
Let $\T$, $\T'$ be subsets of $\flag{\tmod}$ such that
\[ \Int{\Theta'} \supset \Theta.\]

\begin{defn}[Straight-spaced sequences, \cite{Kapovich:2014aa}]
Let $\epsilon\ge 0$ be a number. A (finite, infinite, or biinfinite) sequence $(x_n)$ is called \emph{$(\Theta,\epsilon)$-straight} if, for each $n$, the segments $x_n x_{n+1}$ are $\Theta$-regular and
\[ \angle_{x_n}^\xi (x_{n-1}, x_{n+1}) \ge \pi - \epsilon.\]
Moreover, $(x_n)$ is called \emph{$l$-spaced} if $d(x_n,x_{n+1}) \ge l$ for all $n$.
\end{defn}

\begin{defn}[Morse sequence]
A sequence $(x_n)$ is called \emph{$(\Theta,D)$-Morse} if the piecewise geodesic path formed by connecting consecutive points by geodesic segments is a $(\Theta, D)$-Morse quasigeodesic.
\end{defn}

\begin{thm}[Morse lemma for straight spaced sequences, \cite{Kapovich:2014aa}]\label{seq_morse}
For $\Theta,\Theta',D$, there exists $l,\epsilon$ such that any $(\Theta,\epsilon)$-straight $l$-spaced sequence $(x_n)$ in $X$ is $D$-close to a parallel set $P(\tau_+,\tau_-)$ of type $\tmod$. Moreover, the nearest point projection $\bar{x}_n$ of $x_n$ on $P(\tau_+,\tau_-)$ satisfies
\[\bar{x}_{n\pm m} \in V(\bar{x}_n,\ttst{\tau_{\pm}}),\quad \forall m \in\mathbb{N}.\]
Furthermore, the sequence $(x_n)$ is a uniform Morse sequence with parameters depending only on the given data $\Theta,\Theta',D$.
\end{thm}

\subsection{Replacements}\label{sec:replacements}

Here we define an alternative notion of stability of quasigeodesics, namely that \emph{Morse property is stable under replacements}. See Theorem \ref{replacements}, and its generalized version Theorem \ref{replacements2}.

\medskip
We first develop an important tool which will be needed in the proof of these results.

\begin{defn}[Longitudinal segments]
 Let $y_1$, $y_2$ be any points in $\pt$. The (oriented) segment $y_1y_2$ is called  \emph{$\T$-longitudinal} if $y_2 \in\ V(y_1,\tst{\tau_+})$. Moreover, $y_1y_2$ is called  \emph{($\tmod$)-longitudinal} if $y_2 \in\ V(y_1,\ost{\tau_+})$
 \end{defn}
 
 Convexity of $\T$-cones implies: 
 
\begin{lem}[Concatenation of longitudinal segments]\label{lem:con_longitudinal}
Let $x_1,x_2,x_3\in \pt$ be points such that $x_1x_2$ and $x_2x_3$ are $\T$-longitudinal. Then $x_1x_2$ is also $\T$-longitudinal.
\end{lem}

\begin{prop}[Nearby diamonds]\label{nd2}
Let $\g: [a,b]\r X$ be an $(L,A,\T,D)$-Morse qusaigeodesic, and let $\delta> 0$ be any number. Let $\pt$ be a parallel set such that the image of $\g$ is contained in $\nbhd{\delta}{\pt}$. Denote the nearest point projection of $\g(t)$ into $\pt$ by $\bg(t)$. 
Suppose that ${\bg(a)\bg(b)}$ is longitudinal. Then,
there exist $R' = R'(L,A,\T,\T',D,\delta)$, $D' = D'(L,A,\T,\T',D,\delta)$ such that the following holds: For any $t_1,t_2 \in [a,b]$,  if $(t_2 - t_1) \ge R'$, then ${\bg(t_1)\bg(t_2)}$ is $\Theta'$-longitudinal and $\gamma([t_1,t_2]) \subset \nbhd{D'}{\tdd{\bg(t_1),\bg(t_2)}}$.
\end{prop}

\begin{proof}
Let $\Theta'',\T'''\subset\tmod$ be auxiliary subsets such that $\Int{\Theta'}\supset\T'''$,  $\Int{\T'''} \supset \T''$, and $\Int{\T''} \supset \T$. Note that when $(b-a)$ is sufficiently large, the \tedl\ asserts that $\bg(a)\bg(b)$ is $\Theta''$-regular, which in turn makes ${\bg(a)\bg(b)}$ $\Theta''$-longitudinal. Therefore, it follows that
\[ \bg([a,b]) \subset \nbhd{c+\delta}{\tdddd{\bg(a),\bg(b)}} \subset  \nbhd{c+\delta}{V\left(\bg(a),\ttttst{\tau_+}\right)},\]
where $c = c(\Theta'',\T''',D+\delta)$ is the constant as in Theorem \ref{cd}. 

Let $t\in [a,b]$ be any point. From above, we get $d(\bg(t), V\left(\bg(a),\ttttst{\tau_+}\right)) \le c+\delta$. Using the \tedl\ again, we obtain that when $t-a \ge R \gg L(c+\delta)$, then
\[ \bg(t) \in V\left(\bg(a),\ttst{\tau_+}\right),\]
i.e. $\bg(a)\bg(t)$ is $\T'$-longitudinal. By reversing the direction of $\g$, we also get that when $b-t \ge R$, then $\bg(t)\bg(b)$ is $\T'$-longitudinal.

For arbitrary $t_1,t_2\in[a,b]$, we let $t = (t_2-t_1)/2$. The same argument applied to the paths $\gamma([a,t])$, $\gamma([t,b])$ implies that when $t - t_1 \ge R$, and $t_2 - t \ge R$, then $\bg(t_1)\bg(t)$ and $\bg(t)\bg(t_2)$ are $\T'$-longitudinal segments. Using Lemma \ref{lem:con_longitudinal}, we get that $\bg(t_1)\bg(t_2)$ is $\T'$-longitudinal.

Therefore, $\bg(t_1)\bg(t_2)$ is $\T'$-longitudinal whenever $t_2-t_1 \ge 2R$.

\medskip
After enlarging $\T'$, the second part follows from Theorem \ref{cd}.
\end{proof}

We now turn to the discussion of replacements.

\begin{defn}[Morse quasigeodesic replacements] 
Let $\gamma : I \rightarrow X$ be an $(L,A,\Theta,D)$-Morse quasigeodesic, and let $[t_1,t_2]$ be a subinterval of $I$. Let $\gamma':{[t_1,t_2]}\rightarrow X$ be another $(L',A',\Theta',D')$-Morse quasigeodesic such that $\gamma|_{\{t_1,t_2\}} = \gamma'|_{\{t_1,t_2\}}$. An \emph{$(L',A',\Theta',D')$-Morse quasigeodesic replacement} of $\gamma|_{[t_1,t_2]}$ is the concatenation of $\gamma|_{I - (t_1,t_2)}$ with $\gamma'|_{[t_1,t_2]}$.
\end{defn}

\begin{thm}[Replacement lemma]\label{replacements}
Uniform Morse quasigeodesic replacements are uniformly Morse.
\end{thm}

\begin{proof}
Suppose that $I = [a,b]$ is some interval. Let $\gamma : I \rightarrow X$ be an $(L,A,\Theta,D)$-Morse quasigeodesic, and let $\rho : I \rightarrow X$ be obtained by replacing of $\gamma|_{[t_1,t_2]}$ by a $(L',A',\Theta',D')$-Morse quasigeodesic $\gamma' : [t_1,t_2] \rightarrow X$. Let $\Theta''$ be  subset of $\ost{\tmod}$ which contains $\Theta$ and $\Theta'$. Replacing the parameters $(L,A,\Theta,D)$ and $(L',A',\Theta',D')$ by $(L'',A'',\Theta'', D'')$, where $L'' = \max\{L,L'\}$, $A'' = \max\{A,A'\}$, $D'' = \max\{ D,D'\}$, and some $\T'' \supset \T\cup\T'$, we may assume that $(L,A,\Theta,D) = (L',A',\Theta',D')$ to begin with.

\medskip
By the definition, there exists a diamond $\td{x_1,x_2}$ with $d(x_1,\g(a))\le D$, $d(x_2,\g(b)) \le D$ such that $\g([a,b]) \subset \nbhd{D}{\td{x_1,x_2}}$. Without loss of generality, we may assume that $x_1\ne x_2$. The diamond $\td{x_1,x_2}$ spans a unique parallel set $\pt$ such that $x_2 \in\V{x_1}{\tau_+}$. We denote the nearest point projections of $\g(t)$ and $\g'(t)$ to $\pt$ by $\bg(t)$ and $\bgp(t)$, respectively.

By the \tedl\ we get that when $(b-a)$ is sufficiently large, $(b-a) \ge C(\Theta,\Theta',D)$, then ${\bg(a)\bg(b)}$ is $\T'$-longitudinal\footnote{the nearest point projection might not send $\g(a)$ (resp. $\g(b)$) to $x_1$ (resp. $x_2$), but sends into a $D$-neighborhood of $x_1$ (resp. $x_2$).}. 
Using Proposition \ref{nd2}, when $(t_2 - t_1) \ge R'$, then ${\bg(t_1)\bg(t_2)} = {\bgp(t_1)\bgp(t_2)}$ is also $\T'$-longitudinal. From Theorem \ref{cd} we get a constant $D'$ such that $\g'([t_1,t_2]) \subset\nbhd{D'}{\pt}$.

We prove that any subpath $\rho |_{[r_1,r_2]}$ is uniformly close to a diamond.
From above, if $(r_2 - r_1)\ge R'$, for $r_1,r_2\in I$, then $\bg(r_2) \in V(\bg(r_1),\ttst{\tau_+})$. This also holds for $\g'$ and $r_1,r_2\in [t_1,t_2]$ for a bigger $R'$ because $\g'([t_1,t_1])$ is $D'$-close to $\pt$, and ${\bgp(t_1)\bgp(t_2)}$ is longitudinal. Also, note that in this case, possibly after enlarging $\T'$,  $\g|_{[r_1,r_2]}$ and $\g'|_{[r_1,r_2]}$ become uniformly close to $\tdd{\bg(r_1),\bg(r_2)}$ and $\tdd{\bg'(r_1),\bg'(r_2)}$, respectively (Theorem \ref{cd}).

Clearly, when both $r_1,r_2$ belong to one of the sets $[a,t_1]$, $[t_1,t_2]$, $[t_2,b]$, then $\rho({[r_1,r_2]})$ is uniformly close to a diamond. Therefore, the following are only nontrivial cases.

\begin{case1}
$r_1\in [a,t_1]$ and $r_2 \in [t_1,t_2]$. 
\end{case1}

In this case, if $(t_1 - r_1)\ge R'$ and $(r_2-t_1)\ge R'$, 
then from the discussion above we get
\[ \bg(t_1) \in V(\bg(r_1),\ttst{\tau_+}), \quad \bg'(r_2) \in V(\bg(t_1),\ttst{\tau_+}).\]
From convexity of cones, it follows that 
\[ \bg'(r_2) \in V({\bg(r_1)},\ttst{\tau_+}).\]
Since $\tdd{\bg(r_1),\bg(t_1)}$ and $\tdd{\bg'(t_1),\bg'(r_2)}$  are subsets of $\tdd{\bg(r_1),\bg'(r_2)}$, $\rho|_{[r_1,r_2]}$ is uniformly close to $\tdd{\bg(r_1),\bg'(r_2)}$. 

Now we prove the quasiisometric inequality for $\rho(r_1)$ and $\rho(r_2)$. Since the points $\bg(r_1)$ and $\bg'(r_2)$ belong to two opposite cones with tip $\bg(t_1)=\bg'(t_1)$,
\[ \angle_{\bg(t_1)}\left(\bg(r_1),\bg'(r_2)\right) \ge \alpha',\]
where $\alpha' = \titsangle(\T',\partial\st{\tmod})$.
Comparing the geodesic triangle $\triangle\left(\bg(r_1),\bg(t_1),\bg'(r_2)\right)$ with a Euclidean one, we get
\begin{equation*} 
d\left(\bg(r_1),\bg'(r_2)\right) 
\ge  \frac{\sin\alpha'}{2}\left(d\left(\bg(r_1),\bg(t_1)\right) + d\left(\bg'(t_1),\bg'(r_2)\right)\right).
\end{equation*}
It follows that
\begin{multline*}
d(\rho(r_1),\rho(r_2)) = d\left(\g(r_1),\g'(r_2)\right)
\ge \frac{\sin\alpha'}{2} \left({d\left({\g(r_1)},{\g(t_1)}\right) + d\left({\g'(t_1)},{\g'(r_2)}\right)}\right)\\ - 2D'(1+\sin\alpha')
\ge \frac{\sin\alpha'}{2L} |r_1-r_2| - \left(4 D' + A\right).
\end{multline*}
In the last inequality, we are using the quasigeodesic data for the paths $\gamma$ and $\g'$.

\begin{case2}
$r_1\in [t_1,t_2]$ and $r_2 \in [t_2,b]$. 
\end{case2}

This case is settled by reversing the direction of $\gamma$ in the case 1.

\begin{case3}
$r_1\in [a,t_1]$ and $r_2 \in [t_2,b]$.
\end{case3}

The quasiisometric inequality for $\rho(r_1)$ and $\rho(r_2)$ is clear, since
\[d\left( \rho(r_1),\rho(r_2) \right) = d\left( \g(r_1),\g(r_2) \right) \ge \frac{|r_1 - r_2|}{L} - A.\]
It remains only to show that the image $\rho([r_1,r_2])$ is uniformly close to a $\T'$-diamond.

We know that $\g([r_1,r_2])$ is $D$-close to a diamond $\td{x_1,x_2}$ satisfying $d(x_i,\g(r_i))\le D$, and that $\g'([t_1,t_2])$ is $D$-close to a diamond $\td{y_1,y_2}$ satisfying $d(y_i,\g'(t_i))\le D$, for $i=1,2$. Since $\g(t_i) = \g'(t_i)$, it follows that the points $y_1$ and $y_2$ are $2D$-close to $\td{x_1,x_2}$. Let $\pt$ be the unique parallel set spanned by $\td{x_1,x_2}$ satisfying $x_2 \in\V{x_1}{\tau_+}$. Then,
\[
y_1y_2 \subset \nbhd{2D}{\pt}.
\]

Let $\bar{y}_1$, $\bar{y}_2$ denote the projections of $y_1,y_2$, respectively, in $\pt$. Note that the points $y_1,y_2$ are $D$-close to $\g([r_1,r_2])$. Using Proposition \ref{nd2}, it follows that when $d(y_1,y_2)$, or equivalently $(t_2-t_1)$, is sufficiently large, then $\bar{y}_1\bar{y}_2$ is $\T'$-longitudinal. In addition, note that the points $\bar{y}_1,\bar{y}_2$ are $4D$-close to the cones $V(x_1,\tst{\t_+}),V({x_2,\tst{\t_-}})$, respectively. Using the \tedl\, it follows that when $d(x_1,\bar{y}_1)$ and $d(x_2,\bar{y}_2)$, or equivalently $(t_1-r_1)$ and $(r_2 - t_2)$, are large enough, then $x_1\bar{y}_1$ and $\bar{y}_2x_2$ are $\T'$-longitudinal. Therefore,
\[ \bar{y}_1\bar{y}_2 \subset \tdd{x_1,x_2}.\]
Using Theorem \ref{cd}, there is a constant $c$ which depends only on $D,\T',\T''$ such that 
\[\tddd{y_1,y_2} \subset \nbhd{c}{\tdd{\bar{y}_1,\bar{y}_2}} \subset \nbhd{c}{\tdd{x_1,x_2}}.\]
Therefore, $\rho([r_1,r_2])$ is $(c+D)$-close to $\tdd{x_1,x_2}$.
\end{proof}

\begin{rem}
The replacement lemma is false if we relax the Morse condition. It is not generally true that a \emph{uniform quasigeodesic replacement} of an (ordinary) quasigeodesic in a $\catk$ space, $k\ge 0$, is a uniform quasigeodesic. See the example below. However, if $k <0$, then the ordinary quasigeodesics are Morse quasigeodesics, so the replacement lemma for ordinary quasigeodesics holds.
\end{rem}

\begin{exmp}
Let $Y = \R^2$, $\g$ be the $x$-axis. For $r\ge 0$, define $\g'_r : [-r,r] \rightarrow \R^2$ as in Figure \ref{fig:ex4.6} which is a $(4,0)$-quasigeodesic. If $\phi_r$ is the replacement, then $\phi_r(2r) = \phi(r - k_r)$, for some number $k_r > 0$ (observe the point $(2r,0)$). However, if $\phi$ is an $(l,a)$-quasigeodesic, then $d(\phi_r(2r) - \phi(r - k_r)) \ge r/l - a$, which is false for large $r$.

\begin{figure}[h]
\begin{center}
\vspace{.2in}
\begin{tikzpicture}
    \node[anchor=south west,inner sep=0] (image) at (0,0,0) {\includegraphics[width=4.5in]{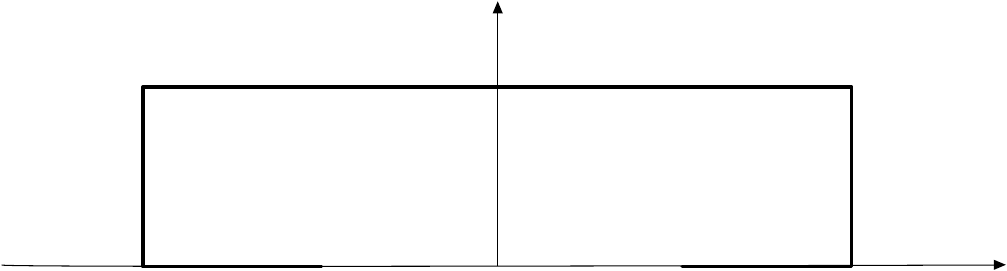}};
    \begin{scope}[x={(image.south east)},y={(image.north west)}]
        \node at (.325,.1) {$ \scriptstyle (-r,0)$};
        \node at (.67,.1) {$\scriptstyle(r,0)$};
        \node at (.088,.1) {$\scriptstyle(-2r,0)$};
        \node at (.89,.1) {$\scriptstyle(2r,0)$};
        \node at (.088,.7) {$\scriptstyle(-2r,r)$};
        \node at (.89,.7) {$\scriptstyle(2r,r)$};
    	\node at (.98,.1) {$x$};
	\node at (.515,.96) {$y$};
	\node at (.3,.76) {$\gamma'$};
    \end{scope}
\end{tikzpicture}
\vspace{.2in}
\end{center}
\caption{}
\label{fig:ex4.6}
\end{figure}
\end{exmp}

We can also replace a Morse quasigeodesic at multiple segments.

\begin{thm}[Generalized replacement lemma]\label{replacements2}
Let $\g : [a,b]\r X$ be an $(L,A,\T,D)$-Morse quasigeodesic, and let $a=t_0 \le t_1 \le \dots\le t_{r_0-1} \le t_{r_0} = b$. For $r = 1,\dots, r_0$, let $\g_r: [t_{r-1},t_r] \r X$ be an $(L',A',\T',D')$-Morse quasigeodesic with $\g_r|_{\{t_{r-1},t_r\}} = \g|_{\{t_{r-1},t_r\}}$. Then the concatenation of  $\g_r$'s is an $(L'',A'',\T'',D'')$-Morse quasigeodesic where $(L'',A'',\T'',D'')$ depends only on $(L,A,\T,D)$ and $(L',A',\T',D')$.
\end{thm}

The proof of this theorem closely follows the proof of the previous one and, we  are omitting the details.

\subsection{Morse subgroups}\label{sec:morsesubgroups}

We first review the notion of Morse subgroups of $G$. See \cite[Sections 7.4, 7.5]{Kapovich:2014aa}.

\medskip
For a finitely generated group $H$ with a finite generating set $A$, we denote by $\C{H,A}$ the associated \emph{Cayley graph} equipped with the word metric. We usually suppress ``$A$'' from the notation, and denote the Cayley graph by $\C{H}$. A finitely generated group $H$ is called \emph{hyperbolic} if $\C{H}$ is hyperbolic.

A metric space $Y$ is called \emph{$(l,a)$-quasigeodesic} if any pair of points can be connected by an $(l,a)$-quasigeodesic. $Y$ is called \emph{quasigeodesic}, if it is $(l,a)$-quasigeodesic for some constants $l,a$. For a finitely generated subgroup $H < G$ and a chosen base point $x\in X$, there is a natural map $o_x :\C{H} \rightarrow X$ induced by the orbit map $H \rightarrow Hx$. A subgroup $H< G$ is called \emph{undistorted} (in $G$), if some (any) $o_x$ is a quasiisometric embedding. General quasiisometric embeddings of a quasigeodesic space into a symmetric space tend to be ``{bad}''. However, one obtains a good control on these embeddings when they are \emph{Morse}; we review this notion below.

\begin{defn}[Morse embeddings, \cite{Kapovich:2014aa}]
Let $X$ be a symmetric space of noncompact type. A map $f:Y\rightarrow X$ from a quasigeodesic space $Y$ is called \emph{$\Theta$-Morse embedding}  if it sends uniform quasigeodesics in $Y$ to uniform $\Theta$-Morse quasigeodesics in $X$. Moreover, $f$ is called \emph{$\tmod$-Morse embedding} if it is $\Theta$-Morse embedding for some $\Theta$. 
\end{defn}

Now we state the notion of Morse subgroups.

\begin{defn}[Morse subgroups, \cite{Kapovich:2014aa}]
A finitely generated subgroup $\Gamma$ of $G$ is called \emph{$\tmod$-Morse}  if, for an(y) $x\in X$, the map $o_x : \C{H}\rightarrow X$ of $\C{\Gamma}$ into $X$ induced by $\Gamma\rightarrow Hx$ is $\tmod$-Morse.
\end{defn}

Note that Morse subgroups are undistorted. 

\medskip
Every $\tmod$-Morse subgroup $\Gamma$ induces a canonical boundary embedding $\beta : \di \Gamma \r \flag{\tmod}$, see \cite{Kapovich:2014aa, MR3736790}. The image of $\beta$ in $\flag{\tmod}$ is called the \emph{flag limit set} of $\Gamma$, and will be denoted by $\limflag{\Gamma}$.

Moreover,   $\tmod$-Morse subgroups are \emph{uniformly $\tmod$-regular} (see \cite{MR3736790}) and, hence, 
the accumulation set in $\di X$ of any orbit $\Gamma x$ contains only points whose types are elements of $\T$, for some compact 
Weyl-convex subset $\T\subset\ost{\tmod}$. The smallest such $\T$ will be denoted by $\T_\Gamma$. 

\begin{rem}
A finitely generated uniformly $\tmod$-regular and undistorted subgroup $\Gamma <G$ is called a \emph{$\tmod$-URU} subgroup. 
The equivalence of $\tmod$-Morse and $\tmod$-URU is the main result of \cite{Kapovich:2014ab}; see also 
\cite{MR3736790}.
\end{rem}

\begin{prop}\label{prop1}
Let $\Gamma$ be a $\tmod$-Morse subgroup, let $\Lambda'$ be any compact set in $\flag{\tmod}$ whose interior contains $\Lambda = \limflag{\Gamma}$, and let $\Theta'$ be any compact set containing $\Theta = \Theta_\Gamma(x)$ in its interior. There exists a number $S>0$ such that any $\g\in \Gamma$ satisfying $d(x, \g x) > S$ also satisfies $\g x\in \VVt{x}$, for some $\tau \in \Lambda'$.
\end{prop}

\begin{proof}
We first prove that there exists $S'> 0$ such that $d(x,\g x) > S'$ implies that $(x,\g x)$ is $\Theta'$-regular. 
Suppose that $S'$ does not exist; then, 
there is an unbounded sequence $(\g_i)_{i\in\mathbb{N}}$ in $\Gamma$ such that $(x, \g_i x)$ is not $\Theta'$-regular for all $i$. 
Then, $(\g_i x)_{i\in\mathbb{N}}$ subconverge to an ideal point whose type $\not\in \Int{\T'}$. This can not happen because the interior of $\Theta'$ contains $\Theta$.

Next we prove that $S$ exists. Assuming that it does not exist, we get an unbounded sequence $(\g'_i)_{i\in\mathbb{N}}$ in $\Gamma$ such that $\g'_i x\not\in \VVt{x}$, for all $i\in\mathbb{N}$ and $\tau\in \Lambda'$.
After extraction we may assume that $(x,\g'_ix)$ is $\Theta'$-regular, for all $i$. 
But then, $(\g'_i x)_{i\in\mathbb{N}}$ does not accumulate in any simplex in the interior of $\Lambda'$ i.e. $\Gamma$ has a limit simplex outside $\Lambda$,
but this gives a contradiction.
\end{proof}

\subsection{Residual finiteness}\label{sec:RF}

An important feature shared by many finitely generated subgroups of $G$ is the \emph{residual finiteness} property which enables us to obtain finite index subgroups which avoid a given finite set of elements.

\begin{defn}[Residual finiteness]
A group $H$ is called \emph{residually finite (RF)} if it satisfies one of the following equivalent conditions: (1) Given a finite subset $S\subset H \setminus  \{1_H\}$, there exists a finite index subgroup $F < H$ such that $F\cap S=\emptyset$. (2) Given an element $h\in H \setminus \{1_H\}$, there exits a finite group $\Phi$ and a homomorphism $\phi : H \rightarrow \Phi$ such that $\phi(h) \ne 1_\Phi$. (3) The intersection of finite index subgroups of $H$ is trivial.
\end{defn}

Residual finiteness of Morse subgroups is a corollary to the following celebrated theorem.

\medskip
Let $R$ be a commutative ring with unity, and let $GL(N,R)$ denote the group (with multiplication) of non-singular $N\times N$ matrices with entries in $R$. Then,

\begin{thm}[A. I. Mal'cev, \cite{MR0003420}]
Finitely generated subgroups of $GL(N,R)$ are RF.
\end{thm}
 
As an application to this theorem, one obtains,

\begin{cor}\label{cor2}
Each finitely generated subgroup $\Gamma < G$  which intersects the center of $G$ trivially is RF.
\end{cor}

\begin{proof}
Under our assumptions, the adjoint representation $\Gamma \r GL({\mathfrak g})$ 
 is faithful.
\end{proof}

For a subgroup $\Gamma < G$, we define the \emph{norm} of $\Gamma$ with respect to $x\in X$ as \[\| \Gamma \|_x = \inf\{ d(x,\gamma x) \mid 1_\Gamma \ne \gamma \in \Gamma\}.\]  Note that when $\|\Gamma\|_x > 0$, $\Gamma$ is discrete. Residual finiteness implies the following useful lemma which we use to obtain subgroups whose nontrivial elements send $x$  arbitrarily far: 

\begin{lem}\label{residual}
Let $\Gamma$ be a RF discrete 
subgroup of $G$. For any $R\in\mathbb{R}$, there exist a finite index subgroup $\Gamma' < \Gamma$ such that $\| \Gamma' \|_x \ge R$.
\end{lem}

\begin{proof}
Since $\Gamma$ is discrete, the set $\Phi = \{\gamma \mid d(x,\gamma x) < R\}$ is finite. The assertion follows from the residual finiteness property.
\end{proof}

Combining this lemma with Proposition \ref{prop1}, we get the following:

\begin{cor}\label{cor1}
Let $\Gamma< G$ be a RF $\tmod$-Morse subgroup, let $\Lambda'$ be any compact set in $\flag{\tmod}$ whose interior contains 
$\Lambda = \limflag{\Gamma}$, and let $\Theta'$ be any compact set containing $\Theta = \Theta_\Gamma$ in its interior. There exists $S_1>0$ such that for any $S\ge S_1$ there exists a finite index subgroup $\Gamma'$ of $\Gamma$ satisfying $\| \Gamma'\|_x > S$ which also satisfies the following:  For any $\g'\in \Gamma'$ exists $\tau \in \Lambda'$ for which $\g'x\in \VVt{x}$. 
\end{cor}

Now we briefly turn to the discussion of pairwise antipodal subgroups before proving our main theorem in the next section.

\begin{defn}[Antipodal Morse subgroups]\label{def:antipodal_morse}
A pair of $\tmod$-Morse subgroups $\Gamma_1,\Gamma_2 < G$ is called \emph{antipodal} if their flag limit sets in $\flag{\tmod}$ are antipodal.
\end{defn}

Let $\Gamma_1, \dots,\Gamma_n$ be pairwise antipodal, RF $\tmod$-Morse subgroups of ${G}$. Let $\Theta \subset \ost{\tmod}$ be a subset which contains the sets $\Theta_{\Gamma_i}$, for $i = 1,\dots,n$, in its interior.

\begin{prop}\label{prop315}
There exists a collection $\{\Lambda'_1,\dots,\Lambda'_n\}$ of pairwise antipodal, compact subsets of $\flag{\tmod}$, and a number $S_2>0$ such that for any $S\ge S_2$ there exists a collection of finite index subgroups $\Gamma_1', \dots,\Gamma_n'$ of $\Gamma_1, \dots,\Gamma_n$, respectively, satisfying $\| \Gamma_1' \|_x\ge S,\dots,\| \Gamma_n' \|_x \ge S$ which also satisfies the following: For each $i=1,\dots,n$, and for each $\g_i \in \Gamma'_i$, there exists $\tau_i\in\Lambda'_i$ such that
\[\gamma_i x \in V(x,\tst{\tau_i}).\]
\end{prop}

\begin{proof}
Once we show that there exists a collection $\{\Lambda'_1,\dots,\Lambda'_n\}$ such that, for each $i$, the interior of $\Lambda'_i$ contains the flag limit set $\Lambda_i$ of $\Gamma_i$, the first part of the proposition follows from the Corollary \ref{cor1}. We may construct $\Lambda'_1,\dots,\Lambda'_n$ as follows: 

\begin{lem} \label{flag_thickening}
Let $\{\Lambda_1,\dots,\Lambda_n\}$ be a collection pairwise antipodal, compact subsets of $\flag{\tmod}$. Then, there exists a collection $\{\Lambda'_1,\dots,\Lambda'_n\}$ of pairwise antipodal, compact subsets of $\flag{\tmod}$ such that each $\Lambda_i$ is contained in the interior of $\Lambda'_i$.
\end{lem}
\begin{proof}
The case $n=2$ can be proved as follows. Let $\Lambda_1,\Lambda_2$ be a pair of antipodal, compact subsets of $\flag{\tmod}$. Then,
\[ \Lambda_1 \times \Lambda_2 \underset{\mathrm{compact}}{\subset} \antiflag \underset{\mathrm{open}}{\subset} \flag{\tmod}\times\flag{\tmod}.\]
There is a open neighborhood of $\Lambda_1 \times \Lambda_2$ in $\antiflag$ of the form $U_1 \times U_2$. In particular, the subsets 
$U_1$ and $U_2$ are antipodal. Then any pair of compact subsets $\Lambda'_1$ and $\Lambda'_2$ of $U_1$ and $U_2$, respectively, containing $\Lambda_1$ and $\Lambda_2$ in their respective interiors, does the job.

We consider now the general case  $n\ge 3$ and let $\{\Lambda_1,\dots,\Lambda_n\}$ be a collection of subsets as in proposition. For $\Lambda_1$, using  the lemma, we find 
a compact neighborhood $\Lambda'_1$ of $\Lambda_1$ which is antipodal to 
the compact 
 \[\bigcup_{k=2}^n \Lambda_k.\]
Then, $\{\Lambda'_1,\Lambda_2,\dots,\Lambda_n\}$ is new collection pairwise antipodal, compact subsets of $\flag{\tmod}$. 
The same argument yields a compact neighborhood $\Lambda_2'$ of $\Lambda_2$ antipodal to $\Lambda_1', \Lambda_3,...,\Lambda_k$. 
We continue inductively. \end{proof}

This completes the proof of the proposition.
\end{proof}

\section{A combination theorem}\label{sec:combination}
In this section, we prove our main result.

\begin{thm}[Combination theorem]\label{main_result}
Let $\Gamma_1, \dots, \Gamma_n$ be pairwise antipodal, RF $\tmod$-Morse subgroups of $G$. Then, there exist finite index subgroups $\Gamma'_i < \Gamma_i$, for $i=1,\dots,n$, such that $\gen{\Gamma'_1, \dots, \Gamma'_n}$ is $\tmod$-Morse, and is naturally isomorphic to $\Gamma'_1*\dots* \Gamma'_n$
\end{thm}

\begin{proof}
We first fix our notations. We denote the $\tmod$-flag limit sets of $\Gamma_1,\dots, \Gamma_n$ by $\Lambda_1,\dots, \Lambda_n$, respectively. Let $\Theta\subset \ost{\tmod}$, let $\{\Lambda'_1,\dots, \Lambda'_n\}$ be a collection of compact, pairwise antipodal subsets in  $\flag{\tmod}$, and let $S_2>0$ be as in  Proposition \ref{prop315}. As always, the point $x$ will be treated as a fixed base point in $X$. Finally, $\Theta \subset \Theta' \subset \Theta''$ are \iinv s of $\ost{\tmod}$ such that
\[ \Int{\Theta''} \supset \Theta',\quad \Int{\Theta'} \supset \Theta.\] 

By Proposition \ref{prop315}, for each $S>S_2$ there exist finite index subgroups $\Gamma'_1,\dots,\Gamma_n'$ of $\Gamma_1,\dots, \Gamma_n$, respectively,  of norms $\|\Gamma_i'\|_x \ge S$, such that for each $i=1,\dots,n$, and each $\gamma_i \in \Gamma_i'$,
\begin{equation}\label{4.5.0}
\gamma_i x \in V(x,\tst{\tau_i}),
\end{equation}
 for some $\tau_i\in\Lambda'_i$.
Let $A_i$ be a finite generating set of $\Gamma'_i$, for each $i = 1,\dots,n$. This choice endows each $\Gamma_i'$ with a word metric, and thus yields a $\Theta$-Morse embedding $o_x^i : \C{\Gamma'_i,A_i} \rightarrow X$ induced by the orbit map $\Gamma'_i \rightarrow \Gamma'_i x$. We take the standard generating set $A = A_1\cup \dots\cup A_n$ of the abstract free product $\Gamma' = \Gamma'_1 *\dots*\Gamma'_n$.
We obtain a natural homomorphism $\varphi: \Gamma' \rightarrow G$. 
When $S$ sufficiently large we prove that $o_x: \C{\Gamma',A} \rightarrow X$ is a $\Theta'$-Morse embedding, i.e. we prove that the geodesics of $\C{\Gamma',A}$ are mapped to uniform Morse quasigeodesics in $X$. This not only will prove that $\varphi$ is injective, 
but also will show that the subgroup $\gen{\Gamma'_1, \dots, \Gamma'_n}$ of $G$ generated by $\Gamma'_1, \dots,\Gamma_n'$ is $\tmod$-Morse.

\begin{claim}
There exists $S_0> 0$ such that if $S \ge S_0$, then the map $o_x: \C{\Gamma',A} \rightarrow X$ sends (finite) geodesics to uniform Morse quasigeodesics.
\end{claim}

\begin{proof}
Given any $\gamma\in \Gamma'$, there is a natural embedding of $\C{\Gamma_i'}$ into $\C{\Gamma'}$ given by the right multiplication map $\gamma_i \mapsto \gamma_i\gamma$. Any geodesic in $\C{\Gamma'}$ is a concatenation of paths which are images of the geodesics under the maps above. By equivariance, it suffices to study the geodesics in $\Gamma'$ starting at $1_{\Gamma'}$. Any geodesic $\psi$ with starting point $1_{\Gamma'}$ in $\C{\Gamma'}$ can be written as
\begin{equation}\label{4.5.1}
\psi~:~1_{\Gamma'},~ \gamma_{k_1},~\gamma_{k_2}\gamma_{k_1},\dots,
\end{equation}
where the path joining $\gamma_{k_r}\gamma_{k_{r-1}}\dots\gamma_{k_1}$ and $\gamma_{k_{r-1}}\dots\gamma_{k_1}$ in $\C{\Gamma'}$ is the image of a geodesic segment in $\C{\Gamma'_i}$ connecting the identity to $\gamma_{k_r}$ under the map $(\cdot) \mapsto (\cdot) (\gamma_{k_{r-1}}\dots\gamma_{k_1})$, assuming that $\gamma_{k_r} \in\Gamma'_i$. We group together $\gamma_{k_r}$'s in above to avoid two consecutive ones coming from same $\Gamma_i$'s.

The (finite) sequence (\ref{4.5.1}) is mapped to $x,~ \gamma_{k_1} x,~\gamma_{k_2}\gamma_{k_1} x, \dots$ under the map $o_x$; to avoid cumbersome notations, we denote $\gamma_{k_r}\gamma_{k_{r-1}}\dots\gamma_{k_1}$ by $g_r$, denote $\gamma_{k_r}\gamma_{k_{r-1}}\dots\gamma_{k_1} x$ by $p_r$ and assume that  the index $r$ of this sequence varies between $0$ and $r_0$. Using these notations, we have
\begin{equation}\label{4.5.2}
g_r = \gamma_{k_r} g_{r-1}, \quad g_r x = p_r.
\end{equation}
  Let $m_1 = p_0$, $m_{r_0} = p_{r_0}$, and, for $2\le r\le r_0 -1$, let $m_r$ denote the midpoint of $p_{r-1}$ and $p_r$ (see Figure \ref{fig:combination1}).

It follows from (\ref{4.5.0}) that all the segments $p_{r-1} p_r$ in $X$ are $\Theta$-regular and have length at least $S$. Moreover, it follows from (\ref{4.5.2}) that, for any $1\le r \le r_0 -1$, precomposing the right multiplication action $g_r^{-1} \curvearrowright \Gamma$ with $o_x$ maps the hinge $p_{r-1} p_r p_{r+1}$ to $(\gamma_{k_r}^{-1} x) (x) (\gamma_{k_{r+1}}x)$ which is of the form $(\gamma_i x)(x)(\gamma_j x)$, for some $\gamma_i \in\Gamma'_i$, $\gamma_j \in\Gamma'_j$, $i\ne j$. From (\ref{4.5.0}), we get that $\gamma_i x \in V(x,\tst{\tau_i})$ and $\gamma_j x \in V(x,\tst{\tau_j})$, for some $\tau_i \in\Lambda_i'$ and $\tau_j \in \Lambda_j'$. To simplify our notation, the corresponding images of $m_r$ and $m_{r+1}$ are denoted by $m'_i$ and $m'_j$, respectively.

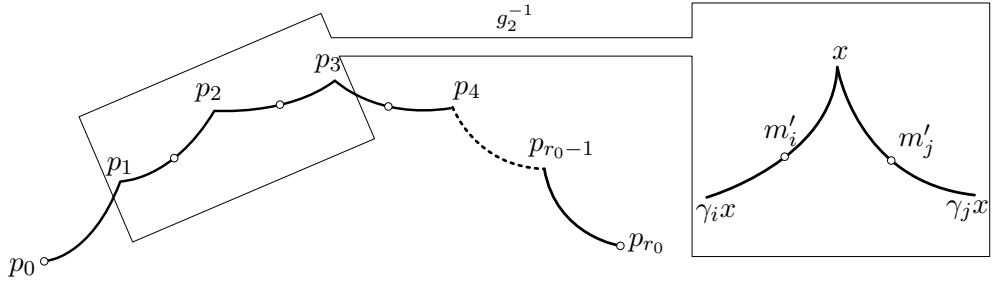
\begin{figure}[h]
\begin{center}
\vspace{.2in}
\begin{tikzpicture}
    \node[anchor=south west,inner sep=0] (image) at (0,0,0) {\includegraphics[width=5in]{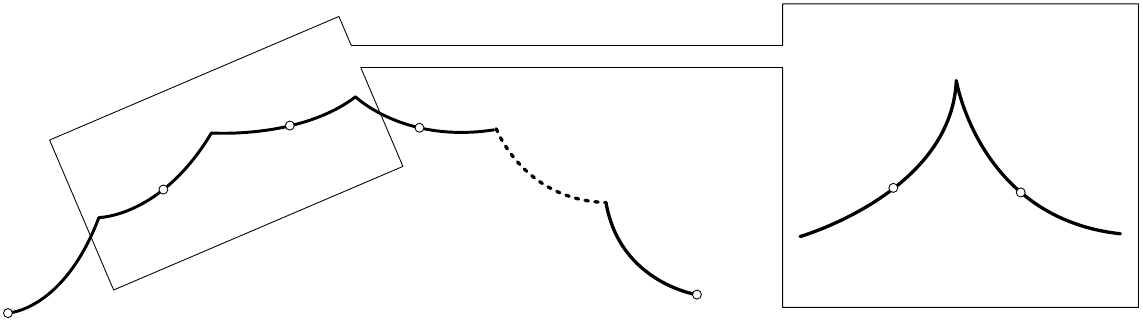}};
    \begin{scope}[x={(image.south east)},y={(image.north west)}]
        \node at (-0.015,0.0) {$p_0$};
        \node at (0.087,.376) {$p_1$};
        \node at (0.18,.653) {$p_2$};
        \node at (0.305,.76) {$p_3$};
        \node at (0.45,.66) {$p_4$};
        \node at (0.549,.445) {$p_{r_0-1}$};
        \node at (0.638,.09) {$p_{r_0}$};
        \node at (0.5,.93) {$\scriptstyle g^{-1}_2$};
        \node at (0.84,.8) {$x$};
        \node at (0.71,.2) {$\gamma_i x$};
        \node at (0.973,.21) {$\gamma_j x$};
        \node at (0.778,.51) {$m_i'$};
        \node at (0.92,.47) {$m_j'$};
    \end{scope}
\end{tikzpicture}
\vspace{.2in}\end{center}
\captionsetup{width=0.7\textwidth}
\caption{Morse embedding of quasigeodesics. The hollow points represent the midpoint sequence $(m_i)$.}
\label{fig:combination1}
\end{figure}

Let $D = \max\{D(\Lambda'_i,\Lambda'_j,x) \mid 1 \le i< j \le n\}$, where $D(\Lambda'_i,\Lambda'_j,x)$ is the constant given by Corollary \ref{D}. Moreover, let $R_1(x,\Lambda'_i,\Lambda'_j,\Theta',\Theta)$ and $f_1 (x,\Lambda'_i,\Lambda'_j,\Theta',\Theta,\xi)$ be the quantities as in Proposition \ref{tool}. Define
\[
 R_1 = \max_{i,j,~{i\ne j}} \left\{R_1(x,\Lambda'_i,\Lambda'_j,\Theta',\Theta)\right\},
\]
and
\[
 f_1 = \max_{i,j,~{i\ne j}} \left\{ f_1(x,\Lambda'_i,\Lambda'_j,\Theta',\Theta,\xi)\right\}.
\]
Note that $d(x,m'_i),d(x,m'_j) \ge S/2$. Using part \ref{tool2} of Proposition \ref{tool}, when $S/2\ge R_1$, then $m'_i m'_j$ is $\Theta'$-regular. Using part \ref{tool3} of the same proposition we get
\[\angle_{m'_i}^\xi(x,m'_j), \angle_{m_j}^\xi(x,m'_i) \le f_1(S/2).\]
Moreover, using and (\ref{4D}), we obtain
\[ d(m'_i,m'_j) \ge (S\sin\alpha)/2 - 4 D,\]
where $\alpha = \titsangle{(\Theta,\partial\T')}$. 
Therefore, when $S \ge 2R_1$, the sequence $(m_r)$ is $\left(\Theta', 2f_1(S/2)\right)$-straight and $\left((S\sin\alpha)/2 - 4 D\right)$-spaced.

For any $\delta'>0$, Theorem \ref{seq_morse} applied to $\Theta'$, $\Theta''$ and $\delta'$ concludes that there exists $S_0\gg R_1$ such that when $S\ge S_0$, then the sequence $(m_r)$ is $\delta'$-close to a parallel set $P(\tau_-,\tau_+)$ such that the nearest point projection map sends $m_r m_{r+1}$ to a $\Theta''$-longitudinal segment. This proves that the piecewise geodesic path $p_0p_1\dots p_{r_0}$ is a  uniform Morse quasigeodesic for sufficiently small $\delta'$.

Finally, we  prove that $o_x\circ \psi$ is uniformly Morse. By invoking the Morse property of $\Gamma'_i$'s, we get that each segment of $o_x\circ \psi$ connecting a consecutive pair $p_r$ and $p_{r+1}$ is a uniform Morse quasigeodesic. Therefore, $o_x\circ \psi$ is obtained by replacing the geodesic segments $p_rp_{r+1}$ of the path $p_0p_1\dots p_{r_0}$ by uniform Morse quasigeodesics. From the generalized replacement lemma (Theorem \ref{replacements2}), it follows that $o_x\circ\psi$ is also a uniform Morse quasigeodesic.
\end{proof}

This completes the proof of the theorem.
\end{proof}

\begin{rem}
The RF condition in the above theorem can be relaxed by integrating the content of Corollary \ref{cor2} into the hypothesis. Precisely, instead of requiring $\Gamma_i$'s to be RF one may require that $\Gamma_i$'s intersect the center of $G$ trivially. When $G\cong \isomid{X}$, this happens automatically, because $\isomid{X}$ is centerless.
\end{rem}

\medskip
Below is a  more general form of Theorem \ref{main_result} which does not involve passing to finite index subgroups, but  instead 
requires  ``sufficient antipodality and sparseness'' of the subgroups $\Gamma_i$. Let 
\[
(\underbrace{\flag{\tmod}\times ... \times \flag{\tmod}}_{n \text{ times}})^{\mathrm{opp}}
\]
denote the subset of $(\flag{\tmod})^n$ consisting of $n$ tuples of pairwise antipodal flags. For a subset 
\[
A\subset (\underbrace{\flag{\tmod}\times ... \times \flag{\tmod}}_{n \text{ times}})^{\mathrm{opp}}
\]
and for $x\in X$ define the subset $O_{A,x}\subset X^n$ consisting of $n$-tuples $(x_1,...,x_n)$ such that 
for some $(\tau_1,...,\tau_n)\in A$, we have $x_i\in V(x, \st{\tau_i})$, $i=1,...,n$.

\begin{thm}\label{thm:general}
For each compact 
\[
A\subset (\underbrace{\flag{\tmod}\times ... \times \flag{\tmod}}_{n \text{ times}})^{\mathrm{opp}},
\]
and $\Theta\subset \ost{\tmod}$, there exists a constant $S=S(A,\Theta,x)$ such that the following holds. Let $\Gamma_1,...,\Gamma_n$ be $P$-Anosov subgroups of $G$ such that:

a. $\| \Gamma_i \|_x\ge S$, $i=1,...,n$. 

b. For each $\gamma_i\in \Gamma_i, i=1,...,n$, the segment $x\gamma_i(x)$ is $\Theta$-regular.  

c. For each $n$-tuple of nontrivial elements $\gamma_i\in \Gamma_i -\{1\}$, $i=1,...,n$, 
 we have 
$$
(\gamma_1(x),...,\gamma_n(x))\in O_{A,x}. 
$$

Then the subgroup of $G$ generated by $\Gamma_1,...,\Gamma_n$ is $P$-Anosov, and is naturally isomorphic to the free product $\Gamma_1 * ... * \Gamma_n$. 
\end{thm}

\begin{proof}
The proof is very similar to the one of Theorem \ref{main_result}. The conclusion of Proposition \ref{prop315} 
now becomes a hypothesis on the subgroups $\Gamma_i$, so no passage to finite index subgroups is required. Hence, the rest 
of the proof of Theorem \ref{main_result} goes through. 
\end{proof}

\begin{rem}
We should note that this theorem is in the spirit of the ``quantitative ping-pong lemma'' of Breuillard and Gelander, see \cite[Lemma 2.3]{BG}. 
\end{rem}

\medskip
As a last remark, we note that the traditional Klein-Maskit combination theorems are stated not in terms of group actions on 
symmetric spaces but in terms of their actions on the sphere at infinity; they also do not involve passing to finite index subgroups. The following is a reasonable combination conjecture in the setting of Anosov subgroups:

\begin{conj}
Let $A_1,...,A_n\subset \flag{\tmod}$ be nonempty compact subsets such that any two distinct elements of 
$$
A:= \bigcup_{i=1}^n A_i
$$
are antipodal. Suppose that $\Gamma_1,...,\Gamma_n$ are $P_{\tmod}$-Anosov subgroups of $G$ such that for all $i=1,...,n$ and all 
$\gamma\in \Gamma_i-\{1\}$ we have
$$
\gamma(A -A_i)\subset \Int{A_i}. 
$$
Then the subgroup $\Gamma$ of $G$ generated by $\Gamma_1,...,\Gamma_n$ is $P_{\tmod}$-Anosov. 
\end{conj}

Note that under the above assumptions, $\Gamma$ is naturally isomorphic to the free product $\Gamma_1*...*\Gamma_n$, see e.g. \cite{tits1972free}.

\bibliographystyle{amsalpha}

\begin{thebibliography}{KLM09}
\addcontentsline{toc}{section}{References}

\bibitem[BC08]{baker2008combination}
M.~Baker and D.~Cooper, \emph{A combination theorem for convex hyperbolic
  manifolds, with applications to surfaces in 3-manifolds}, J. Topol.
  \textbf{1} (2008), no.~3, 603--642.

\bibitem[Ben97]{Benoist}
Y.~Benoist, \emph{Propri\'et\'es asymptotiques des groupes lin\'eaires}, Geom.
  Funct. Anal. \textbf{7} (1997), no.~1, 1--47.

\bibitem[BG08]{BG}
E.~Breuillard and T.~Gelander, \emph{Uniform independence in linear groups},
  Invent. Math. \textbf{173} (2008), no.~2, 225--263.

\bibitem[DK18]{Drutu:2018aa}
C.~Dru{\c{t}}u and M.~Kapovich, \emph{Geometric group theory}, Colloquium
  Publications, vol.~63, American Mathematical Society, 2018.

\bibitem[Ebe96]{eberlein1996geometry}
P.~B. Eberlein, \emph{Geometry of nonpositively curved manifolds}, Chicago
  Lectures in Mathematics, University of Chicago Press, Chicago, IL, 1996.

\bibitem[Git99]{GITIK199965}
R.~Gitik, \emph{Ping-pong on negatively curved groups}, J. Algebra \textbf{217}
  (1999), no.~1, 65--72.

\bibitem[Gro87]{MR919829}
M.~Gromov, \emph{Hyperbolic groups}, Essays in group theory, Math. Sci. Res.
  Inst. Publ., vol.~8, Springer, New York, 1987, pp.~75--263.

\bibitem[GW12]{guichard2012anosov}
O.~Guichard and A.~Wienhard, \emph{Anosov representations: domains of
  discontinuity and applications}, Invent. Math. \textbf{190} (2012), no.~2,
  357--438.

\bibitem[Kle83]{klein1883neue}
F.~Klein, \emph{Neue {B}eitr\"age zur {R}iemann'schen {F}unctionentheorie},
  Math. Ann. \textbf{21} (1883), no.~2, 141--218.

\bibitem[KLM09]{kapovich2009convex}
M.~Kapovich, B.~Leeb, and J.~J. Millson, \emph{Convex functions on symmetric
  spaces, side lengths of polygons and the stability inequalities for weighted
  configurations at infinity}, Journal of Differential Geometry \textbf{81}
  (2009), no.~2, 297--354.

\bibitem[KLPa]{Kapovich:2014aa}
M.~Kapovich, B.~Leeb, and J.~Porti, \emph{Morse actions of discrete groups on
  symmetric space}, arXiv:1403.7671.

\bibitem[KLPb]{Kapovich:2014ab}
\bysame, \emph{A {M}orse {L}emma for quasigeodesics in symmetric spaces and
  euclidean buildings}, arXiv:1411.4176, To appear in Geometry \& Topology.

\bibitem[KLP16]{survey}
\bysame, \emph{Some recent results on {A}nosov representations}, Transform.
  Groups \textbf{21} (2016), no.~4, 1105--1121.

\bibitem[KLP17]{MR3736790}
\bysame, \emph{Anosov subgroups: dynamical and geometric characterizations},
  Eur. J. Math. \textbf{3} (2017), 808--898.

\bibitem[KLP18]{MR3720343}
\bysame, \emph{Dynamics on flag manifolds: domains of proper discontinuity and
  cocompactness}, Geom. Topol. \textbf{22} (2018), no.~1, 157--234.

\bibitem[KLar]{Kapovich:2017aa}
M.~Kapovich and B.~Leeb, \emph{Discrete isometry groups of symmetric spaces},
  Handbook of Group Actions (S-T.~Yau L.~Ji, A.~Papadopoulos, ed.), ALM series,
  International Press, to appear.

\bibitem[Lab06]{labourie2006anosov}
F.~Labourie, \emph{Anosov flows, surface groups and curves in projective
  space}, Invent. Math. \textbf{165} (2006), no.~1, 51--114.

\bibitem[Mal40]{MR0003420}
A.~I. Mal'cev, \emph{On isomorphic matrix representations of infinite groups},
  Rec. Math. [Mat. Sbornik] N.S. \textbf{8 (50)} (1940), 405--422.

\bibitem[Mas65]{maskit1965klein}
B.~Maskit, \emph{On {K}lein's combination theorem}, Trans. Amer. Math. Soc.
  \textbf{120} (1965), 499--509.

\bibitem[Mas68]{maskit1968klein}
\bysame, \emph{On {K}lein's combination theorem. {II}}, Trans. Amer. Math. Soc.
  \textbf{131} (1968), 32--39.

\bibitem[Mas71]{maskit1971klein}
\bysame, \emph{On {K}lein's combination theorem. {III}}, Advances in the
  {T}heory of {R}iemann {S}urfaces ({P}roc. {C}onf., {S}tony {B}rook, {N}.{Y}.,
  1969), Ann. of Math. Studies, No. 66, vol.~66, Princeton Univ. Press,
  Princeton, N.J., 1971, pp.~297--316.

\bibitem[Mas93]{maskit1993klein}
\bysame, \emph{On {K}lein's combination theorem. {IV}}, Trans. Amer. Math. Soc.
  \textbf{336} (1993), no.~1, 265--294.

\bibitem[MP09]{pedroza2009}
E.~Mart\'inez-Pedroza, \emph{Combination of quasiconvex subgroups of relatively
  hyperbolic groups}, Groups Geom. Dyn. \textbf{3} (2009), no.~2, 317--342.

\bibitem[MPS12]{MR2994828}
E.~Mart\'inez-Pedroza and A.~Sisto, \emph{Virtual amalgamation of relatively
  quasiconvex subgroups}, Algebr. Geom. Topol. \textbf{12} (2012), no.~4,
  1993--2002.

\bibitem[Shc13]{MR3003738}
V.~Shchur, \emph{A quantitative version of the {M}orse lemma and
  quasi-isometries fixing the ideal boundary}, J. Funct. Anal. \textbf{264}
  (2013), no.~3, 815--836.

\bibitem[Ste99]{MR1688579}
N.~Steenrod, \emph{The topology of fibre bundles}, Princeton Landmarks in
  Mathematics, Princeton University Press, Princeton, NJ, 1999, Reprint of the
  1957 edition, Princeton Paperbacks.

\bibitem[Sul14]{MR3175245}
H.~Sultan, \emph{Hyperbolic quasi-geodesics in {CAT}(0) spaces}, Geom. Dedicata
  \textbf{169} (2014), 209--224.

\bibitem[Tit72]{tits1972free}
J.~Tits, \emph{Free subgroups in linear groups}, J. Algebra \textbf{20} (1972),
  250--270.

\end{thebibliography}
\providecommand{\bysame}{\leavevmode\hbox to3em{\hrulefill}\thinspace}
\providecommand{\MR}{\relax\ifhmode\unskip\space\fi MR }
\providecommand{\MRhref}[2]{
  \href{http://www.ams.org/mathscinet-getitem?mr=#1}{#2}
}
\providecommand{\href}[2]{#2}

\vspace{2cm}

\Addresses

\end{document}